\documentclass[12pt]{amsart}

\usepackage[numbers,sort]{natbib}
\usepackage{amsmath}
\usepackage[psamsfonts]{amssymb}
\usepackage[all]{xy}
\usepackage{mathtools}
\usepackage{graphicx}
\usepackage[T1]{fontenc}
\usepackage[bitstream-charter]{mathdesign}
\usepackage{hyperref}
\usepackage{color}
\usepackage{dsfont}
\usepackage{accents}
\usepackage{xy}

\let\oldproofname=\proofname
\renewcommand{\proofname}{\rm\bf{\oldproofname}}

\usepackage[symbol]{footmisc}

\newtheorem{theorem}{Theorem}[section]
\newtheorem{prop}[theorem]{Proposition}
\newtheorem{lemma}[theorem]{Lemma}
\newtheorem{cor}[theorem]{Corollary}

\newtheorem{question}[theorem]{Question}
\newtheorem{ex}[theorem]{Example}

\theoremstyle{definition}

\theoremstyle{remark}
\newtheorem{rmk}[theorem]{Remark}

\def\R{\mathbb{R}}

\def\Z{\mathbb{Z}}
\def\N{\mathbb{N}}

\def\C{\mathbb{C}}

\newcommand{\adj}{\operatorname{adj}}
\newcommand{\Crit}{\operatorname{Crit}}

\def\m{{q}}

\def\mrm#1{{\mathrm{#1}}}
\def\cl#1{{\mathcal{#1}}}

\def\bb#1{{\mathbb{#1}}}

\newcommand{\eps}{\epsilon}
\newcommand{\id}{\mathrm{id}}

\newcommand{\del}{\partial}

\begin{document}



\title{Persistent transcendental B\'ezout theorems}



\author{Lev Buhovsky$^1$, Iosif Polterovich$^2$, Leonid Polterovich$^3$, Egor Shelukhin$^4$  and 
Vuka\v{s}in Stojisavljevi\'{c}$^5$}
\date{}

\footnotetext[1]{Partially supported by ERC Starting Grant 757585.}
\footnotetext[2]{Partially supported by NSERC and FRQNT.}
\footnotetext[3]{Partially supported by the Israel Science Foundation
grant 1102/20.}
\footnotetext[4]{Partially supported by NSERC, Fondation Courtois, Alfred P. Sloan Foundation.}
\footnotetext[5]{Partially supported by CRM-ISM postdoctoral fellowship and ERC Starting Grant 851701.}

\maketitle

\begin{abstract}
An example of Cornalba and Shiffman from 1972 disproves in dimension two or higher a classical prediction that the count of zeros of holomorphic self-mappings of the complex linear space should be controlled by the maximum modulus function. We prove that such a bound holds for a modified coarse count inspired by the theory of persistence modules originating in topological data analysis. 
\end{abstract}

\tableofcontents


\section{Introduction and main results}\label{subsec-TBez}
\subsection{The transcendental B\'ezout problem}
The classical B\'ezout theorem states that the number of common zeros of $n$ polynomials in $n$ variables is generically bounded by the product of their degrees.
The {\it transcendental B\'ezout problem} is concerned with  the count of zeros of entire maps
$\C^n \to \C^n$.  It is motivated by a number of influential mathematical ideas. The  starting point is
Serre's famous G.A.G.A. \cite{Serre}, by now understood as a meta-mathematical principle stating that complex {\it projective} analytic geometry reduces to algebraic geometry. A prototypical result is
a theorem of Chow \cite{Chow49}, by which every closed complex submanifold of $\C P^n$ is necessarily algebraic, i.e., is given as the set of solutions of a system of polynomial equations.
However, as the following simple example shows, Chow's theorem fails in the affine setting.
\begin{ex} \label{ex-1-vsp} {\rm Consider an analytic function $f:\C \to \C$ given by
$$f(z)= e^z+1 = (e^x \cos y +1) +ie^x \sin y,\; z=x+iy\;.$$
Zeros of $f$ form an infinite discrete set $\{(2k+1)\pi i, \; k \in \Z\}$.
It is not biholomorphically equivalent to any algebraic (and hence  finite) proper subset of $\C$.}
\end{ex}

In order to revive at least some parts of G.A.G.A. in the affine framework one needs a substitute of the notion of the degree of a polynomial for entire mappings $f:\C^n \to \C^n$. As it is put in \cite{GLO}, {\it  ``A transcendental entire function that can be expanded into an infinite power series can
be viewed as a ``polynomial of infinite degree", and the fact that the degree is
infinite brings no additional information to the statement that an entire function
is not a polynomial."}  To this end, one introduces the {\it maximum modulus}
$$\mu(f,r) = \max _{z \in B_r} |f(z)|\;,$$
where $B_r$ stands for the closed ball of radius $r$.
This quantity has at least two degree-like features. First,
assume that\footnote{Here and throughout the text we denote by $\log$ the logarithm to base $2$.}
$$\limsup_{r \to \infty} \frac{\log \mu(f,r)}{\log r} < k+1.$$
Then, remarkably,  $f$ is a polynomial of the total degree $\leq k$.
This is a minor generalization of Liouville's classical theorem. Thus one can distinguish polynomials in terms of the maximum modulus.

In what follows, let $\zeta(f,r)$ denote the number of zeros of a continuous map  $f: \C^n \to \C^n$
inside the ball $B_r$.

The second feature of the maximum modulus of an entire function $f: \C \to \C$  is given by 
the following statement which readily follows from Jensen's formula:
if $f(0)\neq 0$, then for every $a>1$
\begin{equation}
\label{eq-Jensen-vsp}
\zeta(f,r) \leq C \log \mu(f,ar) \;\;\forall r >0\;,
\end{equation}
where $C$ is a positive constant depending on $a$ and $f(0)$. For instance, in Example \ref{ex-1-vsp}
both $\zeta$ and $\log \mu$ grow linearly in  $r$.

These two features might have given a hope that $\log \mu(f,r)$ is an appropriate substitute of the degree for an entire map $f: \C^n \to \C^n$ (this was known as the transcendental B\'{e}zout problem). However,
this analogy was  overturned  by Cornalba and Shiffman \cite{CornalbaShiffman} who famously constructed, for $n=2$, an entire map $f$ with
$\log \mu(f,r) \leq C_{\epsilon} r^\epsilon$ for every $\epsilon >0$ (and hence of growth {\it order} zero),  with  $\zeta(f,r)$ growing  arbitrarily fast.
As Griffiths wrote in \cite{Griffiths-TB2} {\it ``This is the first instance known to
this author when the analogue of a general result in algebraic geometry fails
to hold in analytic geometry."}

\subsection{Coarse zero count}
One of the motivations for  the present paper is to further explore  the Cornalba--Shiffman example using the notion of {\it coarse zero count} introduced in \cite{BPPPSS}, which is
based on topological persistence. The idea, roughly speaking, is to
discard the zeros corresponding to small oscillations of the map. It turns out that with such a count
we are able to get a  Jensen-type estimate \eqref{eq-Jensen-vsp}, albeit with a possibly non-sharp power of $\log \mu(f,r)$ in the right-hand side,  see \eqref{eq:mainbezout} below.

Given a continuous map $f: \C^n \to \C^n$ and positive numbers $\delta, r>0$,  we define the counting function $\zeta(f,r,\delta)$  of {\it $\delta$-coarse zeros} of $f$ inside a ball $B_r$ as the number of connected components of the set
$f^{-1}(B_\delta)\cap B_r$ which contain zeros of $f$, see Figure \ref{Coarse_vs_Classical}.

\begin{figure}[ht]
	\begin{center}
		\includegraphics[scale=0.5]{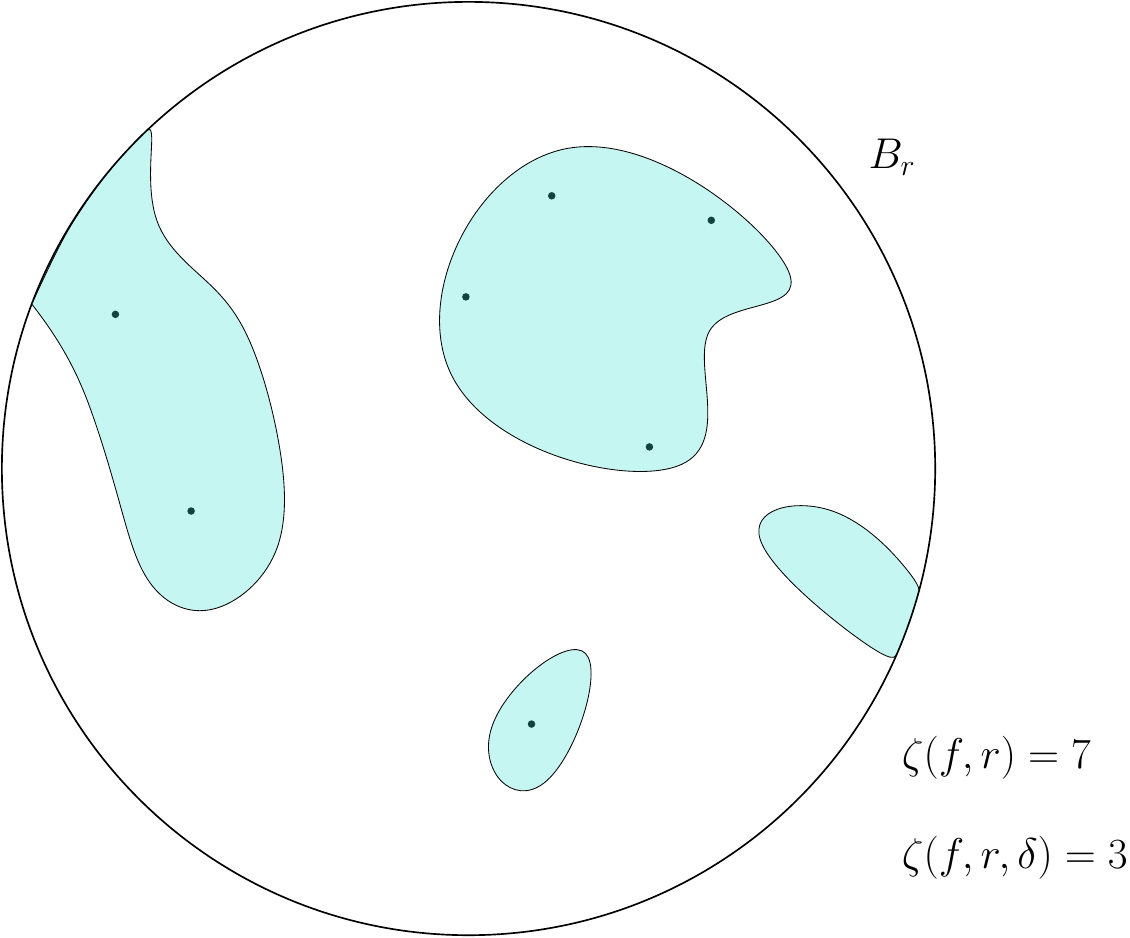}
		\caption{Dots represent zeros of $f$, while shaded regions depict the set $f^{-1}(B_\delta)$}
		\label{Coarse_vs_Classical}
	\end{center}
\end{figure}

\begin{theorem}\label{thm: analytic bezout}
	For any analytic map $f: \C^n \to \C^n$ and any  $a>1$, $r>0$,  and   $\delta \in (0, \frac{\mu(f, ar)}{2})$,  we have
\begin{equation}
\label{eq:mainbezout}
\zeta(f,r,\delta) \leq C \left(\log\left(\frac{\mu(f, ar)}{\delta}\right)\right)^{2n-1},
\end{equation}
	where the constant $C$ depends only on $a$ and $n$. 	
\end{theorem}

This theorem is proved in Section \ref{sec: proof main}. Its generalization in the framework
of topological persistence is presented in Section \ref{sec: persistence}.

Note that by Liouville's theorem, unless $f$ is constant, $\mu(f, ar)$ is unbounded. Therefore, for any given $\delta > 0,$ the condition $\delta \in (0,\mu(f,ar)/2)$  holds for all $r$ large enough.
\begin{rmk}\label{rmk:expn} Consider a higher-dimensional generalization of Example \ref{ex-1-vsp}: take an analytic map $f:\C^n \to \C^n$ given by
$$f(z_1,\dots,z_n) = \left(e^{z_1}+1,\ldots,e^{z_n}+1\right).$$
It is easy to see that $\log \mu(f,r)$ grows linearly in $r$ and $\zeta(f,r,\delta)$ grows as  $r^n$  when  $r \to\infty$, for $\delta$ sufficiently small. It would be interesting to understand whether the power of the logarithm in
\eqref{eq:mainbezout} is sharp or it can be improved, possibly, to $n$.
\end{rmk}
It follows from Theorem \ref{thm: analytic bezout}  that for the Cornalba-Shiffman example  the coarse count of zeros grows slower than any positive power of $r$, see Theorem \ref{Thm:CS_Coarse_Count} below for precise asymptotics.

\begin{rmk}  Consider a function $f$ of growth order $\leq \rho,$ that is,  for all $\eps>0$, there exist positive constants $A_{\eps}, B$, such that    \[|f(z)| \leq A_{\eps} e^{B|z|^{\rho+\eps}}\] everywhere. Then
by \eqref{eq:mainbezout},  $\zeta(f,r,\delta)$ grows slower than $r^{(2n-1)\rho+\eps}$ for every $\eps>0.$
At the same time, it was shown in   \cite[equation (1.9)]{Carlson-TB} that for any $\alpha>0$,
$\zeta(f+c, r)$ grows slower than $r^{(2n-1)\rho+1 +\alpha}$ for almost all $c>0$ small enough. While the growth rate in \eqref{eq:mainbezout} is slightly sharper, it is interesting to note
that the power  $(2n-1)\rho$ appears in both bounds.
\end{rmk}
\subsection{Cornalba--Shiffman example: a coarse perspective} Let us remind the Cornalba--Shiffman construction.
Let $g:\C \to \C$ be given by
$$g(z)=\prod_{i=1}^\infty \left( 1- \frac{z}{2^i} \right).$$
For $k\geq 1$ an integer, let
$$g_k(z)=\frac{g(z)}{1-\frac{z}{2^k}}$$
be the function defined by the same product with $k$-th term excluded. All the infinite products converge uniformly on compact subsets of $\C$ and hence $g$ and $g_k$ are holomorphic by Weierstrass' theorem. For a positive integer $c$ we define a polynomial $P_c:\C \rightarrow \C$ as
$$P_c(w)=\prod_{j=1}^c \left( w- \frac{1}{j} \right).$$
Given a strictly increasing sequence of positive integers $\mathfrak{c}=\{c_i \}$,  $c_1<c_2<\ldots$ define $f:\C^2\to \C$ as
$$f(z,w)=\sum_{i=1}^\infty 2^{-c_i^2}g_i(z)P_{c_i}(w).$$
$f$ converges uniformly on compact sets and is hence holomorphic by Weierstrass' theorem in several variables. Finally, we define a map $F:\C^2 \rightarrow \C^2$, $F(z,w)=(g(z),f(z,w)).$ As shown in \cite{CornalbaShiffman}, for all $\mathfrak{c}$, $F$ is of order zero. 
However, the zero set of $F$ is given by
$$F^{-1}(0)=\left\lbrace \Big(2^i,\frac{1}{j} \Big) ~|~ i=1,2,\ldots ; j=1,\ldots ,c_i \right\rbrace ,$$
as depicted in Figure \ref{CS_Zeros}. The dots represent zeros of $F$ and the number of zeros $\zeta(F,r)=\zeta(F,r,0)$ equals the number of dots inside the circle.

\begin{figure}[ht]
	\begin{center}
		\includegraphics[scale=0.7]{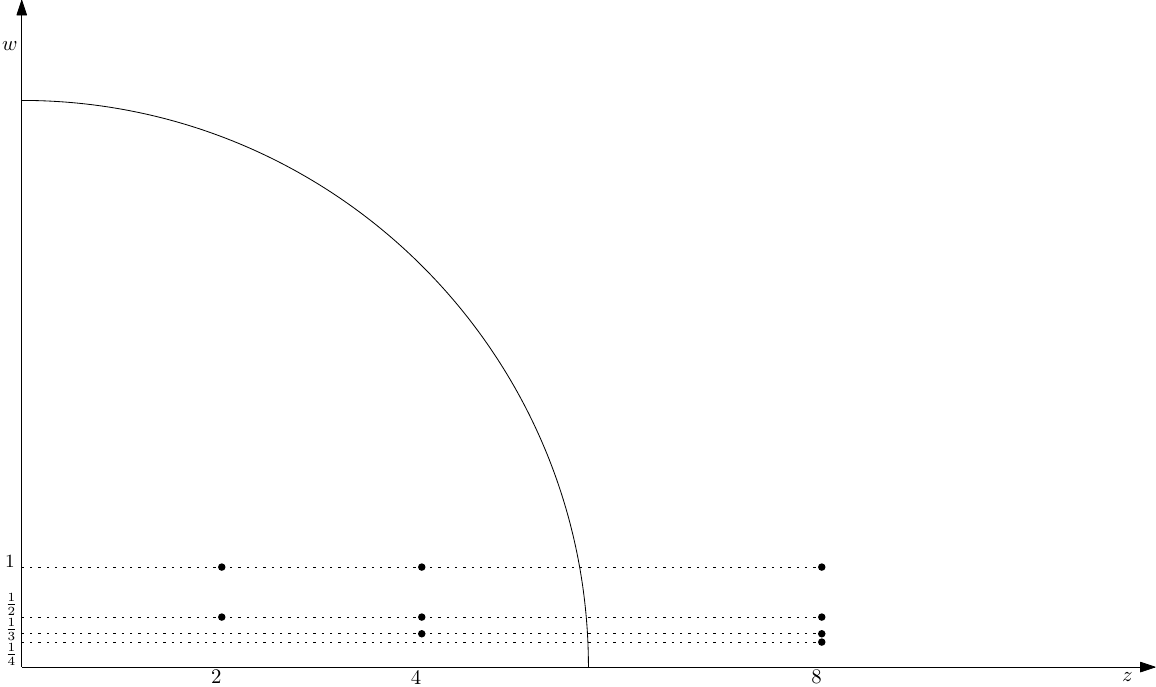}
		\caption{Classical count of zeros}
		\label{CS_Zeros}
	\end{center}
\end{figure}

We now see that by taking $\mathfrak{c}$ which increases sufficiently fast $\zeta(F,r)$ can grow arbitrarily fast which disproves the two-dimensional transcendental B\'{e}zout problem.  More precisely, in \cite{CornalbaShiffman} Cornalba and Shiffman made a remark that $c_i=2^{2^i}$ would suffice. Indeed, it is not difficult to check that for $\lambda>0$, if $c_i=\lfloor2^{\lambda i } \rfloor$ then $\zeta(F,r)=\Theta (r^\lambda)$ i.e. the order of growth of the number of zeros is $\lambda,$ while for $c_i=2^{2^i}$, $\log \zeta(F,r)= \Theta (r)$ and the order of growth of $\zeta(F,r)$ is infinite. Here and further on we write $a(r)=\Theta(b(r))$ if $a(r)=O(b(r))$ and $b(r)=O(a(r))$ as $r \to \infty$; we will also write $a(r) \sim b(r)$ if $\lim_{r\to\infty} a(r)/b(r)=1$.

Let us re-examine the same class of examples from the coarse point of view. The following result is proved in Section \ref{Section:CS_Proofs}. 

\begin{theorem}\label{Thm:CS_Coarse_Count}
Let $\mathfrak{c}$ be an arbitrary increasing sequence of positive integers. When $r\to +\infty$ it holds
$$\log \mu(F, r)=\Theta ((\log r)^2),$$
and for a fixed $\delta>0$
$$\zeta(F,r,\delta) \sim \log r.$$
\end{theorem}

\begin{figure}[ht]
	\begin{center}
		\includegraphics[scale=0.65]{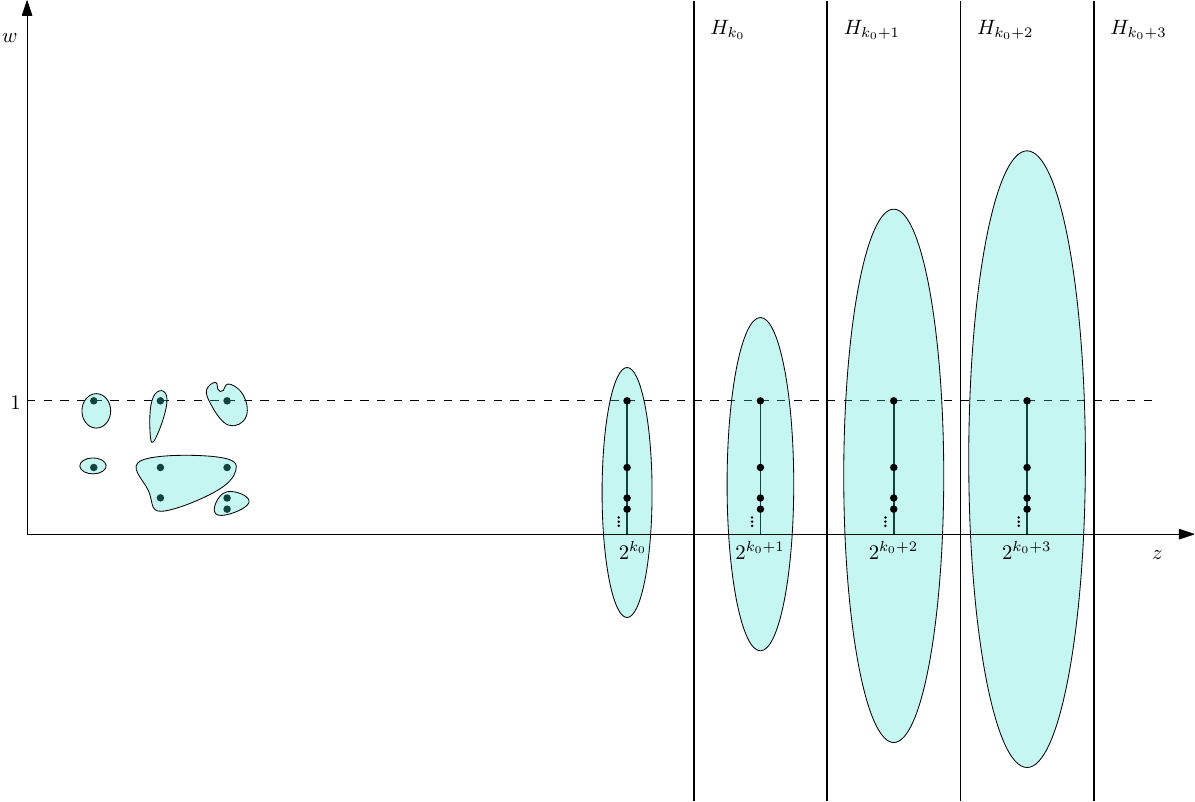}
		\caption{Coarse count of zeros}
		\label{Zeros_Main}
	\end{center}
\end{figure}
Let us  explain the geometric picture behind Theorem \ref{Thm:CS_Coarse_Count}, while referring the reader to Section \ref{Section:CS_Proofs} for detailed proofs. For a fixed $\delta$, we show that the set $\{|F|\leq \delta \}$, while possibly being complicated for small radius $r$, stabilizes for large radii and can be described rather accurately. More precisely, we show that there exists $k_0$, which depends only on $\delta$, such that $\{|F|\leq \delta \}$ contains intervals $\{2^k \} \times [0,1]$ for all $k\geq k_0.$ Thus, for $k\geq k_0$ the zeros on each of the intervals $\{2^k \} \times [0,1]$ are counted coarsely as one zero and the coarse count increases at the rate $\log r.$ This implies $\zeta(F,r,\delta)=O(\log r).$ Furthermore, for $k\geq k_0$, $\{ |F|\leq \delta \}$ will never intersect hyperplanes $H_k= \{ (w,z) ~|~ Re(z)=2^k+2^{k-1} \}.$ In other words, $\{|F|\leq \delta \}$ consists of parts contained between those hyperplanes and which contain intervals $\{2^k \} \times [0,1],$ as shown on Figure \ref{Zeros_Main} (shaded regions represent the set $\{ |F|\leq \delta \}$). This implies that $\zeta(F,r,\delta)= \Theta(\log r),$ which we can improve to $\zeta(F,r,\delta) \sim \log r$ as claimed by Theorem \ref{Thm:CS_Coarse_Count}.

Putting together Theorems   \ref{thm: analytic bezout} and  \ref{Thm:CS_Coarse_Count}, we come to the following conclusion.
It follows from  Theorem \ref{thm: analytic bezout} for $n=2$ that
$$    \zeta(F,r,\delta) \leq C_a (\log \mu(F, ar)-\log \delta)^3 .$$
On the logarithmic scale, this inequality tells us that for fixed $a$ and $\delta$,
\begin{equation}\label{CS_coarse_count}
   \log \zeta(F,r,\delta) = O \left(\log \log \mu(F,a r)\right),
\end{equation}
when $r\to +\infty.$ Theorem \ref{Thm:CS_Coarse_Count} implies that (\ref{CS_coarse_count}) is asymptotically sharp as
\begin{equation}\label{Equality_coarse_count}
\log  \zeta(F,r,\delta) = \Theta (\log \log \mu(F, ar)).
\end{equation}
In other words, Cornalba--Shiffman examples exhibit highly oscillatory behaviour on small scales, which increases the count of zeros in an uncontrollable way and contradicts the transcendental B\'{e}zout problem. However, equality (\ref{Equality_coarse_count}) shows that if we discard small oscillations, the same examples behave essentially as predicted by the coarse version of the transcendental B\'{e}zout problem.

\subsection{Islands vs peninsulas} A connected component of $f^{-1}(B_\delta)\cap B_r$ is called an {\em island} if it is disjoint from $S_r=\partial B_r$ and a {\it peninsula} otherwise. We prove the following result in Section \ref{sec: zeros and islands}: this is a combination of Corollary \ref{cor:islands 3} and Corollary \ref{cor:islands 2}.

\begin{theorem}\label{thm: islands}
Every island has non-empty interior which contains at least one  zero of $f.$
\end{theorem}

Let $\zeta^0(f,r,\delta)$  denote the number of islands, and let   $\tau(f,r,\delta)$ denote the total number of zeros of $f$ with multiplicities {\em contained in islands} in $B_r.$ Since an island can contain  more than one zero,  clearly \[ \zeta^0(f,r,\delta) \leq \tau(f,r,\delta). \]

The following result is a consequence of Rouch\'e's theorem for analytic mappings, see
Section \ref{sec: counting islands}.  

\begin{theorem}\label{prop: rouche}
	For all $a>1$, $r >0$, and $\delta \in (0, \mu(f, ar)/2)$

	\begin{equation}
\label{eq:islandscount}
\tau(f,r,\delta) \leq C_1 \big(\log(\mu(f, ar)/\delta)\big)^{n},
\end{equation}
	where $C_1$ depends only on $a$ and $n$.
\end{theorem}
Note that in view of Remark \ref{rmk:expn},  estimate  \eqref{eq:islandscount} is sharp.

\begin{rmk}
Estimates analogous to \eqref{eq:islandscount} for the usual count of zeros have been proven under positive lower bounds on the Jacobian of $f$ in \cite{Ji,LiTaylor-TB,Li-TB}. Upper bounds in Theorems \ref{thm: analytic bezout} and \ref{prop: rouche} apply to all holomorphic mappings $f: \C^n \to \C^n.$ A detailed comparison of our results with those in \cite{Ji,LiTaylor-TB,Li-TB} is carried out in Section \ref{sec: CA example}. In particular, we give a different proof of a result from \cite{LiTaylor-TB} using Theorem \ref{prop: rouche}.
\end{rmk}

In general, $\zeta(f,r,\delta)$ and $\zeta^0(f,r,\delta)$ can behave rather differently. Indeed, this is the case for the  Cornalba-Shiffman example, as we discuss below.

Let $F$ be a Cornalba-Shiffman map defined previously. It is natural to ask what is the possible growth of the coarse count of islands $\zeta^0(F,r,\delta).$  We show that, as opposed to $\zeta(F,r,\delta)$, $\zeta^0(F,r,\delta)$ can grow arbitrarily slow, with an upper bound depending on $\mathfrak{c}.$ More precisely, in Section \ref{sec: counting islands} we prove the following theorem.

\begin{theorem}\label{Proposition_Islands_Count}
Let $\lambda, \delta >0,$ $l\geq 1$ an integer and denote by $\exp_2(x)=2^x.$ If $c_i=\lfloor\underbrace{\exp_2 \ldots \exp_2}_{l \text{ times}} (\lambda i ) \rfloor$ then there exists a constant $m_{l,\lambda,\delta}$ such for all $r\geq \underbrace{\exp_2 \ldots \exp_2}_{l+1 \text{ times}}(1)$ it holds
$$\zeta^0(F,r,\delta)\leq \frac{1}{\lambda}\underbrace{\log \ldots \log}_{l+1 \text{ times}} r + m_{l,\lambda,\delta}.$$
In particular, for $c_i=2^{2^i}$ as in \cite{CornalbaShiffman}, we have that $\zeta^0(F,r,\delta)=O(\log \log \log r).$
\end{theorem}

From the geometric perspective, the slow growth of $\zeta^0(F,r,\delta)$ is due to elongation of $\{|F|\leq \delta\}$ in the $w$-direction. Namely, as $r$ increases, new groups of zeros of roughly the same modulus $r$ appear, while the components of $\{|F|\leq \delta\}$ which contain these zeros grow in the $w$-direction faster than $r$ (the diameter of their $w$-projection grows faster than $r$). Hence, it takes larger $r$ for a component of $\{|F|\leq \delta \}$ to be fully contained in $B_r$ i.e. to contribute to $\zeta^0(F,r,\delta).$

\subsection{Discussion}

Below we discuss some extensions of our results as well as directions for further research.

\subsubsection{Analytic mappings from $\C^n$ to $\C^k$}\label{sec: discussion zeta m}

A bound analogous to Theorem \ref{thm: analytic bezout} holds for entire mappings $f: \C^n \to \C^k.$ It appears as Theorem \ref{thm: analytic bezout higher homol} in Section \ref{sec: higher}. To have geometric meaning in this case, the definition of $\zeta(f,r,\delta)$ should be generalized. To this end, we look at coarse homology groups of the zero set: for $0 \leq d \leq 2n-1,$ set \[\zeta_d(f,r,\delta) = \dim \mrm{Im} \big( H_d(\{f = 0\} \cap B_r) \to H_d(\{|f|\leq \delta\} \cap B_r) \big).\] Considering generic algebraic maps $f,$ we expect only $0 \leq d \leq n-k$ to have geometric significance. Of particular interest is $d = n-k,$ since this is the dimension where vanishing cycles appear. We prove the upper bound \[\zeta_d(f,r,\delta) \leq C \big(\log(\mu(f,ar)/\delta) \big)^{2n},\]
under the same assumptions on the parameters $a,r,\delta$ as in Theorem \ref{thm: analytic bezout}. 

\subsubsection{Affine varieties}
It would be interesting to generalize our main results to more general affine algebraic varieties $Y \subset \C^N$. The starting case would be varieties which compactify to smooth projective varieties $X \subset \C P^N$ by a normal crossings divisor $D = X \setminus Y.$ We expect that the methods of \cite{Cornalba-Griffiths} combined with those of \cite{BPPPSS}, in particular the subadditivity theorem for persistence barcodes, should be useful for this purpose. See also Section \ref{sec: persistence}.

\subsubsection{Harmonic mappings}
We expect that an analogue of Theorem \ref{thm: analytic bezout} should hold in the context of harmonic maps. Namely, suppose that $h=(h_1,\ldots,h_d): \R^d \to \R^d$ is a harmonic map in the sense that $h_j:\R^d \to \R$ is harmonic for all $j$ (this is equivalent to the map being harmonic in the variational sense \cite[Section 2.2, Examples 3, 4]{HeleinWood} where $\R^d$ is endowed with the standard Euclidean metric). In this case, we expect that the coarse counts $\zeta(h,r,\delta), \zeta^0(h,r,\delta)$ being defined analogously to the above, satisfy the upper bound \begin{equation}\label{eq: h bezout} \zeta^0(h,r,\delta) \leq \zeta(h,r,\delta) \leq C_2 \big(\log(\mu(h, C_1 ar)/\delta) \big) ^{d} \;, \end{equation} for all $a>1$, $r >0$, and $\delta \in (0, \mu(h,C_1 ar)/2),$ where $C_1 \geq 1$ depends on $d$ only, and $C_2$ depends only on $a$ and $d.$ (By Liouville's theorem, the condition on $\delta$ holds for all $r$ large enough, if our mapping is not constant.) This bound is sharp asymptotically in $r$ for all fixed $\delta > 0,$ as can be seen from the example $h = (h_1, \ldots, h_d)$ where $h_i(x_1,\ldots,x_d) = (e^{x_{i+1}} \sin(x_{i}))$ for $1 \leq i \leq d-1$ and $h_d(x_1,\ldots,x_d) = (e^{x_{1}} \sin(x_{d})).$ Note that $\log \mu(h,ar)$ is closely related to the notion of the doubling index of the harmonic function (see for example \cite[Equation (12)]{Logunov-YauLower}).

An outline of the argument is as follows. First, replace Proposition \ref{prop: complex poly} by an analogous estimate for real polynomials of degree $k$ on a ball in $\R^d$ in terms of $k^d$. Second, replace Proposition \ref{lma: Cauchy harmonic} by suitable Cauchy estimates, which hold for entire harmonic maps (see for example \cite[Theorem 2.4]{HFT}, \cite[Chapter 2.2, Proof of Theorem 10]{Evans}). The rest of the argument follows our proof of Theorem \ref{thm: analytic bezout} directly. Note that the only property required from a harmonic mapping is that it satisfies Cauchy's estimates, hence inequality \eqref{eq: h bezout} should extend to a certain ``quasi-analytic" class of mappings.

It would be interesting to realize this outline and to optimize the constant $C_1 \geq 1.$


\subsubsection{Near-holomorphic mappings}

Let us also note that the proof of Theorem \ref{thm: analytic bezout} yields the following stronger result about the coarse count of zeros of continuous functions that are close to holomorphic ones.

\begin{cor}
Fix $b<1$ and $\delta>0$, and let  $h: \C^n \to \C^n$ be a continuous function such that there exists a holomorphic function $f: \C^n \to \C^n$ with $d_{C^0}(h,f) < \frac{b}{2} \delta$.  Then \[ \zeta(h,r, (1+b)\delta) \leq C  \big(\log(\mu(h, a r)/\delta)\big)^{2n-1}\;,\] where $C$ depends on $a,b,n$ only.
\end{cor}

\subsubsection{A dynamical interlude}
A dynamical counterpart of the transcendental B\'{e}zout problem is the
count of periodic orbits of entire maps $f: \C^n \to \C^n$. Here by a $k$-periodic
orbit we mean a fixed point of the iteration $f^{\circ k}= f \circ \cdots \circ f$ ($k$ times).
There exists a vast literature on the orbit growth of algebraic maps $f$
(see e.g. \cite{Artin-Mazur}). For instance, it follows from the B\'{e}zout theorem  that if the components of $f$ are generic polynomials of degree $\leq d$, the number of $k$-periodic orbits does not exceed $d^{kn}$. Can one expect a bound on the number of $k$-periodic orbits in the ball of radius $r$ in terms of the maximum modulus function $\mu(f,r)$? The naive answer is ``no" due to the Cornalba-Shiffman examples. Nevertheless, Theorem \ref{thm: analytic bezout} 
above readily yields
such a bound on the  {\it coarse count} $\zeta (f_k,r,\delta)$, where
$$f_k(z):= f^{\circ k}(z) -z\;.$$
One can check that the maximum modulus function behaves nicely under the
composition and the sum:
$$\mu( f \circ g, r) \leq \mu( f, \mu(g,r)),  \;\;\; \mu(f+g,r) \leq \mu(f,r)+\mu(g,r)\;.$$
Fix $a>1$ and $\delta >0$, set $\tilde \mu(r):= \mu(f, r)$, and put
$$\mu_k(r) = \tilde \mu^{\circ k}(r)+r\;.$$
By Theorem \ref{thm: analytic bezout} we have the desired estimate
\begin{equation} \label{eq-dynamics-main}
\zeta (f_k,r,\delta) \leq C \max\left( \left( \log \left( \frac{\mu_k(a r)}{\delta}\right) \right)^{2n-1}, 1\right)\;.
\end{equation}

A few questions are in order.

\begin{question} Can one find a transcendental entire map $f$ for which
estimate \eqref{eq-dynamics-main} is sharp?
\end{question}

\noindent
A natural playground for testing this question are transcendental H\'{e}non maps
whose entropy as restricted to a family of concentric discs grows arbitrarily fast
\cite{Arosio-etal}.

Further, recall that a $k$-periodic orbit of an entire map $f: \C^n \to \C^n$
 is called {\it primitive} if it is not $m$-periodic with
$m < k$. Denote by $\nu_k(f,r)$ the number of primitive $k$-periodic orbits lying in the ball of radius $r$.

\begin{question} Does there exist a transcendental entire map $f$ of order $0$ (i.e., the modulus $\mu(f,r)$ grows slower than $e^{r^\epsilon}$ for every $\epsilon >0$) such that
$\nu_k(f,r)$ grows arbitrarily fast in $k$ and $r$?
\end{question}

For instance, taking $f(z) = F(z)+z$, where $F$ is a Cornalba-Shiffman map, we see that $\nu_1(f,r)$ can grow arbitrarily fast.  Can one generalize this construction to $k \geq 2$?

\medskip

Finally, let us mention that the failure of the transcendental B\'{e}zout theorem
appears as one of the substantial difficulties in the work \cite{Gutman-etal} dealing
with a dynamical problem of a completely different nature, namely with embeddings of $\Z^k$-actions into the shift action on the infinite dimensional cube (see (2) on p. 1450 in \cite{Gutman-etal}).  In particular, the authors analyze the structure of zeroes of so-called  {\it tiling-like band-limited maps} (see p.1477 and Lemma 5.9). It would be interesting to
perform our coarse count of zeroes (i.e., to calculate $\zeta$) for this class of examples.


%
%

\section*{Organization of the paper} In Section \ref{sec: proof main} we prove Theorem \ref{thm: analytic bezout} providing our solution of the persistent transcendental B\'ezout problem. This result is extended
in section \ref{sec: higher} to maps $\C^n \to \C^m$ and higher homology groups, and in Section \ref{sec: persistence}
it is reformulated and generalized in the context of topological persistence. Section \ref{sec: zeros and islands}
contains a proof of Theorem \ref{thm: islands} on the structure of islands. 

In Sections \ref{Section:CS_Proofs} and  \ref{sec: counting islands} we study the Cornalba-Shiffman example and prove Theorems \ref{Thm:CS_Coarse_Count},\ref{prop: rouche}, and \ref{Proposition_Islands_Count}.

A comparison of our results with an earlier work of Li and Taylor \cite{LiTaylor-TB} on the transcendental B\'ezout problem can be found in Section \ref{sec: CA example}.

\section*{Acknowledgments}
We thank Misha Sodin and Steve Oudot for illuminating discussions, and
the anonymous referee for numerous useful comments.

\section{Proofs}
\subsection{Proof of Theorem \ref{thm: analytic bezout}}\label{sec: proof main}


In order to prove Theorem \ref{thm: analytic bezout} we will approximate an analytic map by a polynomial.  To this end, we first recall a version of the classical Cauchy estimates for complex analytic mappings.

\begin{prop}\label{lma: Cauchy harmonic}
	Let $f: \C^n \to \C^m$ be a complex analytic mapping, $a>1,$ and $R_k = f-p_k$ be the Taylor remainder for the approximation of $f$ by the Taylor polynomial mapping $p_k$ at $0$ of degree $<k.$ Then for all $r>0, k \geq 0,$ \[ \mu(R_k,r)  \leq C_{a} a^{-k} \mu(f,ar)\] for the constant $C_{a} = \frac{a}{a-1}$ depending only on $a.$
\end{prop}

We give a proof for clarity.

\begin{proof}
Let $v \in S^{2n-1} \subset \C^n$ and $u \in B_r(\C) \subset \C.$ Write a point $z$ in $B_r = B_r(\C^n)$ as $z = u v.$ Then $R_k(z) = f(uv)-p_k(uv).$ Now take $v' \in S^{2m-1}$ and set $g(u) = \langle f(uv),v'\rangle , q_k(u) = \langle p_k(uv), v' \rangle.$ Then $q_k$ is the Taylor polynomial of $g$ of degree $<k$ and $r_k = g - q_k$ is the corresponding Taylor remainder. It is enough to bound $r_k(u)$ uniformly in $v,v'$ for $u \in B_r(\C).$

Let $0<r<\rho.$ We use the following integral formula for the Cauchy remainder, see \cite[pages 125-126]{Ahlfors}: for $u,w$ with $|u|=r,$ $|w|=\rho$ \[ r_k(u) = \frac{1}{2\pi i} \int_{S^1_{\rho}} g(w) \Big(\frac{u}{w}\Big)^k \frac{1}{w-u} dw.\]  

Therefore \[ |r_k(u)| \leq \Big(\frac{r}{\rho}\Big)^k \frac{\rho}{\rho - r} \mu(g,\rho)\] and picking $\rho = ar,$ we get \[ |r_k(u)| \leq C_a a^{-k} \mu(g,ar) \leq C_a a^{-k} \mu(f,ar).\]



So taking maxima over $u, v$ and $v',$ we obtain \[ \mu(R_k,r) \leq C_a a^{-k} \mu(f,ar). \]

\end{proof}



Next, we highlight topological properties of polynomials needed for the proof of Theorem \ref{thm: analytic bezout}. 

\begin{prop}\label{prop: complex poly}
	Let $k \geq 1.$ Let $p_1,\ldots, p_n: \C^n \to \C$ be complex polynomials of degree at most $k.$ Set $p:\C^n \to \C^n,$ $p = (p_1,\ldots,p_n)$ for the induced polynomial mapping and assume $p$ is proper. 
	Let $B$ be a closed ball in $\C^n$, denote by $ h= |p|^2 $ and by $\Crit(h|_{\partial B})$ the set of critical points of $h|_{\partial B} .$ Then $p^{-1}(0)$ is a finite set of at most $k^n$ points, while $\Crit(h|_{\partial B})$ has at most  $5k(10k)^{2n-2} $ connected components.
\end{prop} 

\begin{proof}
First,  we estimate the number of connected components of $\Crit(h|_{\partial B}).$ If $h|_{\partial B}$ is constant the claim immediately follows. Otherwise, there exists a point $N \in \del B$ which is regular for $h|_{\del B}.$ Consider complex linear coordinates $(z_1,\ldots, z_n),$ $z_j = x_{2j-1}+i x_{2j},$ in which $B$ is the unit ball, so $\del B$ is the unit sphere, and $N$ is the base vector $(0,\ldots,0,i).$ Note that $h$ is a real polynomial of degree at most $2 k$ in $x_1,\ldots, x_{2n}.$ Following \cite{Non19}, we consider inverse stereographic projection $\theta:\R^{2n-1} \to \del B \setminus \{N\},$ $\theta(u_1,\ldots,u_{2n-1}) = (x_1,\ldots,x_{2n}),$ $x_j = \frac{2u_j}{|u|^2+1}$ for $1 \leq j \leq 2n-1,$ $x_{2n} = \frac{|u|^2-1}{|u|^2+1}.$ Then $\theta^*h = \frac{q}{(|u|^2+1)^{2k}}$ for a polynomial $q$ in $u_1,\ldots,u_{2n-1}$ of degree at most $4k.$ The critical points of $\theta^*h$ are in bijection with those of $h|_{\del B},$ and are given by $2n-1$ polynomial equations $\del_{u_j} q(u) (|u|^2+1) - 4k u_j q(u)= 0,$ $1 \leq j \leq 2n-1,$ each one of degree at most $4k+1 \leq 5k.$ By an estimate of Milnor \cite[Theorem 2]{Milnor} we obtain that the total Betti number of $\Crit(h|_{\partial B})$ is bounded by $5k(10k)^{2n-2}.$ Since 0-th Betti number counts path components, we get that the total number of path components of $\Crit(h|_{\partial B})$ is not greater than $5k(10k)^{2n-2}.$ Finally, the number of connected components is less or equal than the number of path components\footnote{In fact, in this case these two numbers are equal since $\Crit(h|_{\partial B})$ can be triangulated (see Remark \ref{rmk: Semialgebraic_triangulation}) and hence is locally path connected.} and the claim follows.

For the second part of the proposition, it is enough to notice that since $p$ is proper $p^{-1}(0)$ is compact and thus consists of a finite set of points by \cite[Theorem 14.3.1]{Rudin80}. The number of these points is bounded by $k^n$ in view of B\'ezout's theorem (see  \cite[Example 8.4.6]{Fulton-intersection} for example). 
\end{proof}

\begin{rmk}
As can be seen from the proof, the exponent $2n-1$ in Proposition \ref{prop: complex poly} is a boundary effect. We currently do not know if it can be improved.
\end{rmk}

Lastly, we formulate a lemma which connects Proposition \ref{prop: complex poly} to the coarse count of zeros of near-polynomial maps.

\begin{lemma}\label{lema: Critical_Components}
Let $f: \C^n \to \C^n$ be a continuous map, $p:\C^n \to \C^n$ a proper complex polynomial map, $h=|p|^2,B\subset \C^n$ a closed ball and assume that $|f(z)-p(z)|<\delta /2$ for all $z\in B.$ Let $\Omega$ be a connected component of $f^{-1}(B_\delta)\cap B$ which contains a zero of $f.$ Then $\Omega$ contains either a zero of $p$ or a connected component of $\Crit(h|_{\partial B}).$
\end{lemma}
\begin{proof}
First we prove that $\Omega$ contains a local minimum of $h.$ Let $z_0$ be a zero of $f$ in $\Omega.$ Then $ |p(z_0)| \leq |f(z_0)| + |p(z_0)-f(z_0)| < \delta/2 $. It follows that $ \min_\Omega |p| < \delta/2 $, 
and let $ z_1 \in \Omega $ be a point where this minumum is achieved. Then $ |f(z_1)| \leq |p(z_1)|+|f(z_1)-p(z_1)| < \delta/2 + \delta/2 = \delta $, hence $ \Omega $ contains a neighborhood of $ z_1 $ in $ B_r $, and therefore $ z_1 $ is in fact a local minimum of the function $ |p| $, and hence also of the function $h$, on $ B_r.$

Now, if $z_1\in \partial B$ then $z_1 \in \Crit(h|_{\partial B})$ and let $Z$ be a connected component of $\Crit(h|_{\partial B})$ which contains $z_1.$ Since $h$ is constant on $Z$ we have that for every $z\in Z$ it holds $|f(z)|\leq |p(z)|+|f(z)-p(z)|<\delta/2+\delta/2=\delta.$ Thus, $z\in f^{-1}(B_\delta)\cap B$ and therefore $Z\subset f^{-1}(B_\delta)\cap B.$ We claim that in fact $Z\subset \Omega.$ Indeed, $Z\cap \Omega \neq \emptyset$ and since both $Z$ and $\Omega$ are connected, so is $Z\cup \Omega.$ On the other hand, $Z\cup \Omega \subset f^{-1}(B_\delta)\cap B$ and since $\Omega$ is a connected component of $f^{-1}(B_\delta)\cap B$, $Z\subset \Omega$ as claimed.

In case $z_1 \in B \setminus \partial B,$ we have that $p(z_1)=0$ because $p$ is an open mapping.  Indeed, since $p$ is proper, it is a finite mapping, and as it is equidimensional it is therefore an open mapping \cite[Proposition 3, Section 2.1.3]{DAngelo}. 
\end{proof}

\begin{proof}[Proof of Theorem \ref{thm: analytic bezout}]

By Proposition \ref{lma: Cauchy harmonic}, we can approximate $ f $ by a Taylor polynomial mapping $ p $ at $ 0 $ of degree $ <k $ such that $ |f-p| \leq C_{a} a^{-k} \mu(f,ar) $ on $ B_r $. Here we choose $ k $ to be the minimal positive integer such that $ C_a a^{-k} \mu(f,ar) < \delta/2 $. We can slightly perturb $ p $ to make it proper by adding a homogeneous polynomial of degree $k$, while $ |f-p|<\delta/2 $ continues to hold on $ B_r.$ Indeed, $p$ being proper is equivalent to $|p(z)|\to \infty$ as $|z|\to \infty$, which can be achieved by such a perturbation. Lemma \ref{lema: Critical_Components} implies that $\zeta(f,r,\delta)$ is bounded from above by the total number of zeros of $p$ and connected components of $\Crit(h|_{\partial B}).$ Now, Proposition \ref{prop: complex poly} gives us that $\zeta(f,r,\delta)\leq k^n+5k(10k)^{2n-2}$ which proves the claim.
\end{proof}

\begin{rmk}\label{rmk: Dimension_Reduction}
Theorem \ref{thm: analytic bezout} also holds for analytic maps $f:\C^n \to \C^m$ with $m<n.$ Indeed, if $m<n$ we can include $\iota:\C^m\to \C^n$ and define $\tilde{f}=\iota \circ f.$ Now $\zeta(f,r,\delta)=\zeta(\tilde{f},r,\delta)$,  $\mu(f,ar)=\mu(\tilde{f},ar)$ and (\ref{eq:mainbezout}) for $f$ follows from the same inequality for $\tilde{f}.$
\end{rmk}


\subsection{Higher-dimensional counts}\label{sec: higher}

A similar approach leads to a proof of the following more general statement, albeit with a slightly weaker exponent on the right hand side. See also Section \ref{sec: persistence}. Consider the invariants $\zeta_d(f,r,\delta)$ from Section \ref{sec: discussion zeta m}.

\begin{theorem}\label{thm: analytic bezout higher homol}
	For any analytic map $f: \C^n \to \C^m$, $m\leq n$ and any  $a>1$, $r>0$,  an integer $ 0 \leq d \leq 2n -1$, and   $\delta \in (0, \frac{\mu(f, ar)}{2})$,  we have
\begin{equation}
\label{eq:mainbezout2}
\zeta_d(f,r,\delta) \leq C \left(\log\left(\frac{\mu(f, ar)}{\delta}\right)\right)^{2n},
\end{equation}
	where the constant $C$ depends only on $a$ and $n$. 	
\end{theorem}

\begin{proof}
Firstly, we notice that it is enough to prove the theorem in the case $m=n$ by the same reasoning as in Remark \ref{rmk: Dimension_Reduction}. 

By Proposition \ref{lma: Cauchy harmonic}, we can approximate $ f $ by a Taylor polynomial mapping $ p $ at $ 0 $ of degree $ <k $ such that $ |f-p| \leq C_{a} a^{-k} \mu(f,ar) $ on $ B_r $. Here we choose $ k $ to be the minimal positive integer such that $ C_a a^{-k} \mu(f,ar) < \delta/2 $. Since $ ||f|-|p|| \leq |f-p| < \delta/2 $ there exist natural maps
$$ H_d(\{f=0\} \cap B_r) \xrightarrow{i_1} H_d(\{|p| \leq \delta/2\} \cap B_r) \xrightarrow{i_2} H_d(\{|f|\leq \delta \} \cap B_r).$$
It follows that
\begin{equation*}
\begin{gathered}
\zeta_d(f,r,\delta) = \dim Im ( i_2 \circ i_1) \leq \dim H_d(\{|p|\leq \delta/2\} \cap B_r) \\
=  \dim H_d(\{ |p|^2 \leq \delta^2/4\} \cap B_r).
\end{gathered}
\end{equation*}
Identifying $\C^n=\R^{2n}$ we have that $|p|^2$ is a real polynomial of degree $<2k$ and $\{ |p|^2 \leq \delta^2/4\}\cap B_r$ is defined by two polynomial inequalities 
$$\frac{\delta^2}{4}-|p|^2\geq 0 \text{ and } r-x_1^2-\ldots -x_{2n}^2\geq 0.$$
By \cite[Theorem 3]{Milnor},  we have that $\dim H_d(\{ |p|^2 \leq \delta^2/4\} \cap B_r)$ is bounded by $\frac{1}{2}(4k+2k)(3+2k)^{2n-1} $ which finishes the proof.
\end{proof}

\begin{rmk}\label{rmk: Semialgebraic_triangulation}
In the proofs of Theorem \ref{thm: analytic bezout higher homol} and Proposition \ref{prop: complex poly} we used results of Milnor which relate to Betti numbers defined using  \v{C}ech cohomology. These Betti numbers coincide with ones coming from singular homology,  as explained in \cite{Milnor}, due to the fact that semialgebraic sets admit triangulations.
\end{rmk}

\subsection{Zeros and islands}\label{sec: zeros and islands}

In this section we prove Theorem \ref{thm: islands}, which is a combination of Corollary \ref{cor:islands 3} and Corollary \ref{cor:islands 2}.
It is a structure theorem for islands of analytic mappings. We start with the following general result which is possibly known, but we could not locate it in the literature.
%

\begin{prop}\label{prop:islands}
Let $g: U \to \C^m,$ for an open set $U \subset \C^n,$ for $n \geq m \geq 1,$ be a holomorphic mapping. Then every point $q\in U$ of non-degenerate minimum of $h=|g|^2$ must be a zero of $g.$ 
 \end{prop}

\begin{proof}
Write $g = u + iv,$ where $u, v$ are the real and imaginary part of $g$ respectively and $i$ is the imaginary unit. Then $h = |u|^2 + |v|^2.$ Now $d_q h = 2\langle u, d_q u \rangle + 2 \langle v, d_q v \rangle,$ while $d_q g = d_q u + i d_q v.$  This shows that if $q$ is a critical point of $h,$ then $g(q) \in \C^m$ is orthogonal to the image of $d_q g.$ Therefore, as $n \geq m,$ if $g(q) \neq 0,$ then the kernel $K = \ker d_q g$ is a non-trivial complex-linear subspace of $\C^n.$ Now, as $d_q g$ vanishes on $K,$ writing the second order Taylor approximation \[g(z) = g(q) + \frac{1}{2} d^2_q g(z-q, z-q) + o(|z-q|^2)\] of $g$ at $q,$ where $d^2_q g$ is the complex Hessian (or quadratic differential) of $g$ at $q,$ we obtain that \[h(z) = h(q) + \mathrm{Re}\langle d^2_q g(z-q,z-q), g(q)\rangle + o(|z-q|^2)\] is the second order Taylor approximation of $h$ at $q,$ where the brackets denote the Hermitian inner product on $\C^m.$ Hence the Hessian $d^2_q h$ is given on $a,b \in K$ by \[d^2_q h(a,b) = 2\, \mathrm{Re}\langle d^2_q g(a,b), g(q)\rangle.\] Since $d^2_q g$ is a $\C^m$-valued complex bilinear form, $\langle d^2_q g, g(q)\rangle$ is a $\C$-valued complex bilinear form. Therefore, by a classical observation (see \cite[Assertion 1, p.39]{Milnor-book}), the quadratic form of $d^2_q h|_K$ has zero signature as the real part of a complex quadratic form. Therefore it cannot be positive-definite. This contradicts the hypothesis that $q$ is an interior non-degenerate minimum. Hence $g(q) = 0.$\end{proof}



\begin{cor}\label{cor:islands 2}
Let $g: B_r \to \C^m$ be a holomorphic mapping on a ball $B_r$ in $\C^n,$ $n \geq m.$ Consider an island of $\{ |g| \leq \delta \}$ with non-empty interior $V.$ Then $V$ contains a zero of $g.$
\end{cor}

\begin{proof}
Denote by $K$ the island of $\{ |g|\leq \delta \}$ with interior $V.$ Firstly, we claim that $|g|$ is not constant on $K.$ To this end, let $q\in V$, $L$ a complex line passing through $q$ parametrized by $z\in \C$ and denote by $u(z)=\sum_{j=1}^n |g_j(z)|^2$ the restriction of $|g|^2$ to $L.$ If $|g|$ was constant on $K$, $u$ would be constant on a neighbourhood of $q$ in $L$ and we would have that
$$0=\Delta u = 4 \sum_{j=1}^n \frac{\partial g_j}{\partial z} \frac{\partial \bar{g}_j }{\partial \bar{z}}= 4 \sum_{j=1}^n \Big{|} \frac{\partial g_j}{\partial z} \Big{|}^2.$$
This implies that each $g_j$ is constant on a neighbourhood of $q$  and thus constant on the whole $L\cap B_r.$ Since $L$ was taken arbitrarily, we conclude that $g$ is constant, which is impossible since $\{ |g| \leq \delta \}$ has an island by the assumption.

Now, assume by contradiction that $0<\min_K |g|< \max_K |g|=\delta.$ Let $0<2\varepsilon < \min\{ \min_K |g| , \max_K |g| - \min_K |g| \}$ and assume that $g':K\to \C^m$ satisfies $|g-g'|<\varepsilon$ on $K.$ Let $q,q' \in K$ for which $|g(q)|=\min_K |g|$, $|g'(q')|=\min_K |g'|.$ Since 
$$|g(q')|<|g'(q')|+\varepsilon\leq |g'(q)| +\varepsilon <|g(q)| + 2 \varepsilon = \min_K |g| +2\varepsilon < \max_K |g|,$$
we have that $q' \in V$ because $|g|$ equals $\delta$ on $K\setminus V.$ Thus $|g'|$ has an interior local minimum $q'$. By means of a transversality argument we show the following  proposition at the end of this subsection.
\begin{prop}\label{clm: transv} Assume $K$ contains no zeros of $g.$ Then, there exists an open set $U$, $K\subset U$ such that for every $\varepsilon>0$ there exists a holomorphic function $g':U \to \C^m$ such that $|g-g'|<\varepsilon$ on $U$ and $|g'|^2$ is Morse. 
\end{prop}
For $g'$ given by Proposition \ref{clm: transv}, $q'$ must be a zero by Proposition \ref{prop:islands}. However, since $2\varepsilon< \min_K |g|$ we have that $0=|g'(q')|> |g(q')|-\varepsilon>\frac{1}{2} \min_K |g|>0$, which is a contradiction. 
\end{proof}

\begin{cor}\label{cor:islands 3}
Let $g: B_r \to \C^m$ be a holomorphic mapping on a ball $B_r$ in $\C^n,$ $n \geq m.$ Then every island of $\{ |g| \leq \delta \}$ has non-empty interior. 
\end{cor}

\begin{proof}
Let $K$ be an island with empty interior. Then $|g| = \delta$ on $K$ and we denote by $K_i,i\geq 1$ a connected component of $\{|g|\leq \delta + 1/i \} \cap B_r$ which contains $K.$

Firstly, we claim that $\cap_{i\geq 1} K_i =K.$ Indeed, $K \subset \cap_{i\geq 1} K_i$ by definition. On the other hand, since $\{ K_i \}$ is a nested sequence of connected compact sets, $\cap_{i\geq 1} K_i$ is also compact and connected, see \cite[Corollary 6.1.19]{Engelking89}. Since $|g|\leq \delta$ on $\cap_{i\geq 1} K_i$, we have that $\cap_{i\geq 1} K_i$ is a connected subset of $\{ |g| \leq \delta \}$ which contains $K$ and thus has to be equal to $K$ since $K$ is a connected component.

Secondly, we claim that there exists $i_0$ such that for all $i \geq i_0$, $K_i$ are disjoint from $\partial B_r.$ Indeed, if this was not the case, there would exist $x_i \in K_i \cap \partial B_r$ and by compactness of $\partial B_r$ we could assume $x_i \to x_\infty \in \partial B_r.$ This is not possible since $x_\infty \in \cap_{i\geq 1} K_i = K$ which is disjoint from $\partial B_r.$

Thus $\{K_i \}, i\geq i_0$ is a sequence of islands in $B_r$ each of which has a non-empty interior (since it contains an open neighbourhood of $K$). By Corollary \ref{cor:islands 2}, there exists a sequence $z_i\in K_i, i\geq i_0$ of zeros of $g$,  and by compactness we may assume that $z_i\to z_\infty \in K.$ This contradicts the fact that $|g|=\delta$ of $K.$ \end{proof}

We note that the condition $n \geq m$ is essential for Proposition \ref{prop:islands} and Corollaries \ref{cor:islands 2}, \ref{cor:islands 3}, as can be seen from the mapping $g:B_r \to \C^2,$ $B_r \subset \C,$ $g(z) = (1,z).$ Similarly, so is non-degeneracy in Proposition \ref{prop:islands}, as shown by the example $g: B_r \to \C^2,$ $B_r \subset \C^2,$ $g(z,w) = (z,1-z).$ 

\begin{proof}[Proof of Proposition \ref{clm: transv}]


Denote by $g_z$ the complex differential of $g$ with respect to $z=(z_1,\ldots, z_n)$.  Let $U$ be such that the closure $\overline{U} \subset B_r$ and $|g| >c > 0$, $|g_z| < C$ on $U$ for some $c,C.$ We argue as follows in the spirit of Thom's parametric transversality theorem.


The function $|g|^2= \bar{g}g$ is Morse on $U$ whenever $\bar{g} g_z$ is transverse to $0$ as a mapping $U \to \C^n.$ Indeed, in this case by the proof of Proposition \ref{prop:islands}, at every critical point of $|g|^2,$ its Hessian will be the real part of a non-degenerate complex-valued symmetric bilinear form $2\bar{g} g_{zz},$ and as such, it will be non-degenerate. 

Denote by $A_i$ the standard basis in the space $\mrm{Mat}(m \times n,\C)$ of $m \times n$ complex matrices.
Choose $\alpha >0$ small enough, and consider an $\alpha$-net $b_j$ in a compact in $\C^m$ containing $A_i \left(U \right)$ for all $i$.
Put $$v_{ij}(z) = A_iz-b_j\;.$$ We look at the perturbation
$$G(z,\epsilon) = g(z) + \sum \epsilon_{ij} v_{ij}(z)\;.$$
The derivative of $\bar{G}G_z$ over $\epsilon_{ij}$
is
\begin{equation}
\label{eq-1}
A_i\bar{g} + \overline{(A_iz-b_j)}g_z\;.
\end{equation}
Fix any point $z \in U$ and choose $j$ so that the second summand in \eqref{eq-1} is smaller than
$C\alpha$. At the same time at least one of the components of $\bar{g}(z)$
has absolute value $> c'$, where $c'$ depends only on $c$ and $n$.
Thus, the collection $\{A_i \bar{g}\}$ contains elements $c_1e_1, \dots, c_me_m$ where
$e_1, \dots, e_m$ is the standard basis of $\C^m$, and $|c_i| > c'$. As invertible matrices form an open set in $\mrm{Mat}(m \times m, \C),$ after the perturbation $\overline{(A_iz-b_j)}g_z$ of norm at most $C\alpha$ these vectors still span $\C^m$, provided $\alpha$ is small enough. It follows that $\bar{G}G_z$ is transverse to $0$. Thus by Thom's parametric transversality theorem, for almost all $\epsilon$ we have that $|G_\epsilon|^2$ is Morse. Choosing $g' = G_{\epsilon}$ for $\epsilon$ sufficiently small then finishes the proof.
\end{proof}


\section{Coarse analysis of the Cornalba-Shiffman example}\label{Section:CS_Proofs}

The goal of this section is to prove Theorem  \ref{Thm:CS_Coarse_Count}.
We will break down its proof into Propositions \ref{Estimate_of_Max}, \ref{Prop_Upper_Zeta} and \ref{Prop_Lower_Zeta}. Since there are no zeros of $F$ when $r< 2$ we always assume $r\geq 2.$ The following elementary estimate will be used repeatedly. For an integer $m\geq 1,$ \[2^{m(m+1)/2}< \prod_{i=0}^m (1+2^i) < 2^{(m+1)(m+2)/2}.\]

Firstly, we estimate $\mu(F,r)$ as needed for the first part of Theorem \ref{Thm:CS_Coarse_Count}.

\begin{prop}\label{Estimate_of_Max} For all $\mathfrak{c}$ and all $r\geq 2$ it holds
$$ \frac{1}{2} (\log r)^2 - \frac{3}{2}\log r +1 \leq \log \mu(F,r) \leq \frac{3}{2} (\log r)^2 + \frac{7}{2} \log r +C,$$
where $C=4+\log \left( \prod_{i=1}^\infty (1+2^{-i}) \right).$
\end{prop}
\begin{proof}
Let $k\geq 1$ be an integer such that $2^k\leq r <2^{k+1}.$ In this case $k \leq \log r < k+1.$ 
To prove the first inequality it is enough to take $(z,w)=(-r,0).$ Now
$$|F(-r,0)|\geq |g(-r)|\geq |g(-2^k)|=\prod_{i=1}^\infty (1+2^{k-i}),$$
and
$$\prod_{i=1}^\infty (1+2^{k-i})>\prod_{i=1}^k (1+2^{k-i})=\prod_{j=0}^{k-1}(1+2^j)>2^{\frac{k(k-1)}{2}}> \left(\frac{r}{2} \right)^{\frac{k-1}{2}}.$$
Since $\log r<k+1$ we have that $\frac{k-1}{2}>\frac{\log r -2}{2}$ and hence $|F(-r,0)|\geq \left(\frac{r}{2} \right)^{\frac{\log r}{2}-1}.$ Applying logarithm to both sides proves the first inequality.

To prove the second inequality we firstly notice that if $|z|\leq r$ then
$$|g(z)|\leq \prod_{i=1}^\infty (1+2^{-i}|z|)<\prod_{i=1}^\infty (1+2^{k+1-i})=\prod_{j=0}^k (1+2^j) \cdot \prod_{i=1}^\infty (1+2^{-i}).$$
Denoting $C_1=\prod_{i=1}^\infty (1+2^{-i})$ and estimating 
$$\prod_{j=0}^k (1+2^j) \leq \prod_{j=1}^{k+1} 2^j=2^\frac{(k+1)(k+2)}{2}\leq (2r)^\frac{k+2}{2}\leq (2r)^{\frac{\log r}{2}+1},$$
yields
\begin{equation}\label{g-estimate}
|g(z)|< C_1 (2r)^{\frac{\log r}{2}+1}.
\end{equation}
Similarly,
$$|g_i(z)|< C_1 (2r)^{\frac{\log r}{2}+1},$$
for all $i\geq 1.$ Using this inequality we further estimate that if $|(z,w)|\leq r$ then
$$|f(z,w)|\leq \sum_{i=1}^\infty \frac{|g_i(z)| \cdot |P_{c_i}(w)|}{2^{c_i^2}}\leq C_1 (2r)^{\frac{\log r}{2}+1} \sum_{i=1}^\infty \frac{|P_{c_i}(w)|}{2^{c_i^2}}. $$
On the other hand
$$\sum_{i=1}^\infty \frac{|P_{c_i}(w)|}{2^{c_i^2}}=\sum_{i=1}^\infty \frac{\prod_{j=1}^{c_i}|(w-1/j)|}{2^{c_i^2}}<\sum_{i=1}^\infty \frac{(r+1)^{c_i}}{2^{c_i^2}}\leq \sum_{i=1}^\infty \frac{(r+1)^i}{2^{i^2}}.$$
To bound the last term we proceed as follows
$$\sum_{i=1}^\infty \frac{(r+1)^i}{2^{i^2}}=\sum_{1\leq i \leq \log (r+1)} \left( \frac{r+1}{2^i} \right)^i + \sum_{ i > \log (r+1)} \left( \frac{r+1}{2^i} \right)^i <$$
$$< \sum_{1\leq i \leq \log (r+1)} (r+1)^i + \sum_{j=0}^\infty \frac{1}{2^j}< (r+2)^{\log (r+1)}+2<(2r)^{\log 2r}+2.$$
Putting all the inequalities together, we obtain
\begin{equation}\label{f-estimate}
|f(z,w)|< C_1 (2r)^{\frac{\log r}{2}+1} ((2r)^{\log 2r}+2).
\end{equation}
Since $|F(z,w)|=\sqrt{|g(z)|^2+|f(z,w)|^2}$,  combining (\ref{g-estimate}) and (\ref{f-estimate}) proves the desired inequality.
\end{proof}

We will now estimate $\zeta(F,r,\delta)$ from above. Before we carry out the relevant computations, let us explain the geometric intuition behind the estimate. 

Zeros of $F$ belong to intervals $\{ 2^k \} \times [0,1]$, $k\geq 1.$ For a fixed $\delta$, we wish to prove that there exists $k_0$ such that for all $k\geq k_0$, each of the intervals $\{ 2^k \} \times [0,1]$ is fully contained in $\{ |F|\leq \delta \}.$ This is the content of Corollary \ref{Merging_zeros}. Now,  on each of these intervals all zeros belong to the same connected component of $\{|F|\leq \delta\}$ and are thus counted at most once in the coarse count $\zeta(F,r,\delta)$, see Figure \ref{Zeros_Upper}. In other words, each of the intervals $\{ 2^k \} \times [0,1]$, $k\geq k_0$ contributes at most one to $\zeta(F,r,\delta)$ and since they appear at rate $\log r$ we have that
$$\zeta(F,r,\delta)\leq \log r + \text{the error term}.$$
The error term comes from zeros on intervals $\{2^k \}\times [0,1]$ for $k<k_0$ where we can not guarantee merging of zeros in $\{ |F| \leq \delta \}$, i.e. we observe no coarse effects.  Moreover, since $k_0$ depends only on $\delta$, the error terms only depends on $\mathfrak{c}$ and $\delta.$ These considerations are formally proven in Proposition \ref{Prop_Upper_Zeta}.

\begin{figure}[ht]
	\begin{center}
		\includegraphics[scale=0.65]{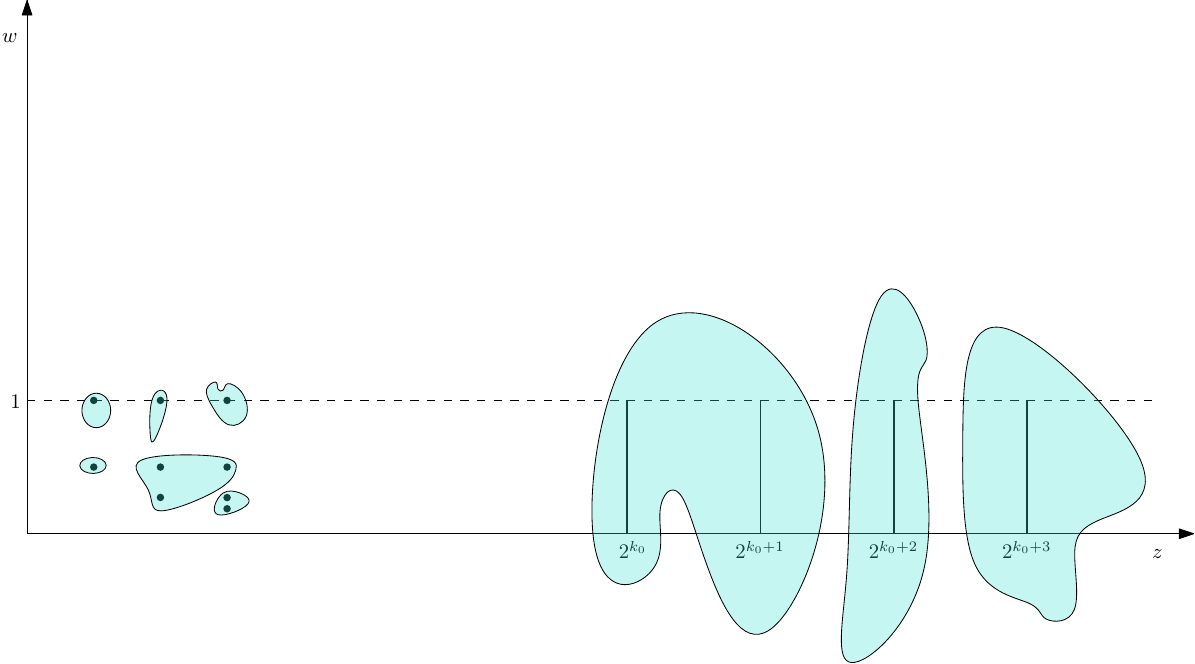}
		\caption{Merging of zeros starting from $2^{k_0}$}
		\label{Zeros_Upper}
	\end{center}
\end{figure}

\begin{lemma}\label{Lemma_Peninsula} For each $\delta \geq 2^{\frac{(i-1)i}{2}-c_i^2}$ the whole interval $\{2^i \} \times [0, 2^{c_i-\frac{(i-1)i}{2c_i}} \delta^\frac{1}{c_i}]$ is contained in $\{ |F|\leq \delta \}.$
\end{lemma}
\begin{proof}
Firstly we notice that
$$|g_i(2^i)| = \prod_{j=1}^{i-1}(2^{i-j}-1) \cdot \prod_{j=i+1}^\infty (1-2^{i-j}) < \prod_{j=1}^{i-1}2^{i-j} = 2^{\frac{(i-1)i}{2}}.$$
Secondly
$$|F(2^i,w)|=|f(2^i,w)|=2^{-c_i^2}|g_i(2^i)| |P_{c_i}(w)|<2^{\frac{(i-1)i}{2}-c_i^2}|P_{c_i}(w)|.$$
Now, if $w\in [0,1]$, $|P_{c_i}(w)|<1$ and the claim follows by the assumption on $\delta$. If $w\in (1, 2^{c_i-\frac{(i-1)i}{2c_i}} \delta^\frac{1}{c_i}]$, $|P_{c_i}(w)|<w^{c_i}$ and thus $|F(2^i,w)|< 2^{\frac{(i-1)i}{2}-c_i^2} w^{c_i} \leq \delta$ and the claim follows.
\end{proof}

\begin{cor}\label{Merging_zeros}
If $\delta\geq 2^{\frac{-i(i+1)}{2}}$ then the whole interval $\{2^i \} \times [0,1]$ is contained in $\{ |F|\leq \delta \}.$
\end{cor}
\begin{proof}
Since $c_i^2\geq i^2$, $\delta\geq 2^{\frac{-i(i+1)}{2}}$ implies that $\delta \geq 2^{\frac{(i-1)i}{2}-c_i^2}$ and thus
$$\{2^i \} \times [0,1]\subset \{2^i \} \times [0, 2^{c_i-\frac{(i-1)i}{2c_i}} \delta^\frac{1}{c_i}] \subset \{ |F|\leq \delta \} $$
by Lemma \ref{Lemma_Peninsula}.
\end{proof}

\begin{prop}\label{Prop_Upper_Zeta} The following estimates hold for $r>2:$
$$\zeta(F,r,\delta) \leq \begin{cases}
   \sum\limits_{1\leq i \leq \log r} c_i , \text{ if } 0<\delta<2^{\frac{-\log r(\log r +1)}{2}}\\
     \log r+2- \sqrt{-2 \log \delta}+ \sum\limits_{1\leq i <\sqrt{-2 \log \delta} } c_i,\text{ if } 2^{\frac{-\log r(\log r +1)}{2}}\leq \delta < \frac{1}{2}\\
     \log r ,\text{ if } \delta \geq \frac{1}{2}\\
    \end{cases}.$$
\end{prop}
\begin{proof}
Firstly, we notice that for all $r\geq 2$
$$\dim H_0(F^{-1}(0) \cap B_r)= \text{ number of zeros in } B_r \leq \sum\limits_{1\leq i \leq \log r} c_i,$$
and hence $\zeta(F,r,\delta) \leq \sum\limits_{1\leq i \leq \log r} c_i$ which proves the first case of the proposition.

Now, we treat the third case, i.e. $\delta \geq \frac{1}{2}.$ In this case, by Corollary \ref{Merging_zeros} all intervals $\{2^i \} \times [0,1], i\geq 1$ are contained in $\{ |F|\leq \delta \}.$ Thus $\dim H_0(\{|F|\leq \delta \} \cap B_r)$ equals the number of intervals $\{2^i \} \times [0,1], i\geq 1$ which intersect $B_r.$ This number is not greater than $\log r$ and thus $ \zeta(F,r,\delta) \leq  \log r.$

Finally, we treat the second case, i.e. $2^{\frac{-\log r(\log r +1)}{2}}\leq \delta < \frac{1}{2}.$ Denote by $r_0>2$ the unique real number such that $\delta =2^{\frac{-\log r_0(\log r_0 +1)}{2}}.$ By the assumption, $r_0\leq r.$ Let $k\geq 1$ be an integer such that $2^k<r_0 \leq 2^{k+1}.$ Now $2^\frac{-(k+1)(k+2)}{2}\leq \delta$ and hence by Corollary \ref{Merging_zeros}, $\{ |F|\leq \delta \}$ contains all interval $\{2^i \}\times [0,1], i \geq k+1.$ Thus
\begin{equation}\label{I+II_bound}
\dim (H_0(\{ |F|\leq \delta \} \cap B_r) \leq \textrm{I} + \textrm{II},
\end{equation}
where 
$$\textrm{I}= \text{ the number of zeros in } B_r \cap \cup_{i=1}^k \{ 2^i \} \times [0,1],$$
and
$$\textrm{II}= \text{ the number of intervals } \{ 2^i \} \times [0,1], i\geq k+1  \text{ which intersect } B_r .$$
From (\ref{I+II_bound}) it follows that 
$$\zeta(F,r,\delta) \leq \textrm{I} + \textrm{II}.$$
Since $k<\log r_0 < \sqrt{-2 \log \delta}$ we have that
$$\textrm{I} \leq \sum_{1\leq i <\sqrt{-2 \log \delta} } c_i .$$
On the other hand,  $r\geq r_0$ and thus $\log r \geq \log r_0 >k$ as well as
$$\textrm{II} \leq \log r -k.$$
Lastly, we use $\sqrt{-2\log \delta}<\log r_0+1 \leq k+2$ to obtain the desired inequality.
\end{proof}

We will now provide a lower bound for $\zeta(F,r,\delta).$ As before, we start by explaining the geometric intuition. 

As $r$ increases,new zeros of $F$ appear on intervals $\{2^k\} \times [0,1]$, i.e. at a rate $\log r.$ We wish to prove that zeros on different intervals will not be counted as one zero in the coarse count $\zeta(F,r,\delta).$  Precisely, in Lemma \ref{Lemma_Separated}, we prove that for a fixed $\delta$, there exists $k_0$ which depends only on $\delta$, such that the set $\{ |F|\leq \delta \}$ does not intersect any of the hyperplanes $H_k=\{ Re(z)=2^k+2^{k-1} \}$ for $k\geq k_0.$ Since $H_k$ separates intervals $\{2^k\} \times [0,1]$ and $\{2^{k+1} \} \times [0,1]$,  $\{|F|\leq \delta \}$ can not contain zeros from different intervals for $k\geq k_0$, see Figure \ref{Figure_Zeros_Lower}. Similarly to the case of the upper bound, this implies that
$$\zeta(F,r,\delta)\geq \log r + \text{the error term},$$
where the error term comes from zeros on intervals $\{2^k \}\times [0,1]$ for $k<k_0$ where we can not guarantee separation of components of $\{ |F| \leq \delta \}.$ These considerations are formally proven in Proposition \ref{Prop_Lower_Zeta}.

\begin{figure}[ht]
	\begin{center}
		\includegraphics[scale=0.65]{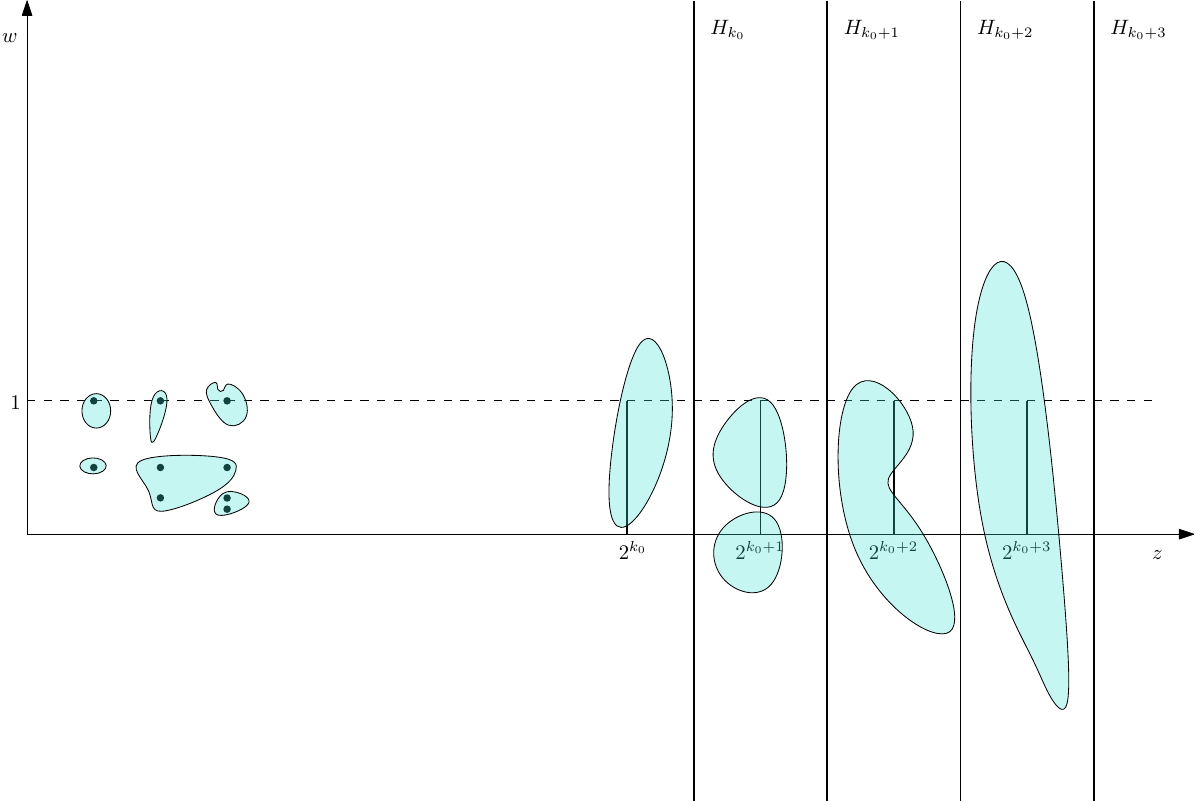}
		\caption{Separation of zeros starting from $2^{k_0}$}
		\label{Figure_Zeros_Lower}
	\end{center}
\end{figure}

In the lemma and the proposition that follow, we denote by $C_0=\frac{1}{2}\prod_{i=1}^\infty \left(1-\frac{3}{2^{i+1}} \right).$

\begin{lemma}\label{Lemma_Separated}
Let $z\in \C$ be such that $Re(z)=2^k+2^{k-1}$ for some integer $k\geq 1.$ Then for all $w\in \C$ it holds
$$|F(z,w)|>C_0 2^\frac{(k-1)k}{2}.$$
\end{lemma}
\begin{proof}
We estimate
$$|F(z,w)|\geq |g(z)| = \prod_{i=1}^\infty |1-2^{-i}z|\geq \prod_{i=1}^\infty |1-2^{-i} Re(z)|=$$
$$=\prod_{i=1}^{k-1} \left(\frac{2^k+2^{k-1}}{2^i}-1 \right)\cdot \frac{1}{2} \cdot \prod_{i=k+1}^\infty \left(1-\frac{2^k+2^{k-1}}{2^i} \right)>C_0 2^\frac{(k-1)k}{2}.$$
\end{proof}

\begin{prop}\label{Prop_Lower_Zeta} For all $r\geq 2$ it holds
$$\zeta(F,r,\delta) \geq \begin{cases}
  \lfloor \log r \rfloor -1, \text{ if } \delta\leq C_0 \\
    \lfloor \log r \rfloor -\sqrt{2\log \delta - 2 \log C_0} -2,\text{ if } C_0< \delta \leq C_0 r^\frac{\log r -1}{2}\\
    \end{cases}.$$
\end{prop}
\begin{proof}
We first prove the case $\delta \leq C_0.$ By Lemma \ref{Lemma_Separated}, we see that on each hyperplane $\{ Re(z)=2^i+2^{i-1} \}, i\geq 1$ it holds $|F(z,w)|>C_0\geq \delta.$ Hence $\{ |F|\leq \delta \}$ does not intersect any of these hyperplanes and in particular zeros $(2^i,1),i\geq 1$ all belong to different connected components of $\{ |F|\leq \delta\}.$ In other words
$$\zeta(F,r,\delta)\geq \text{the number of points }  (2^i,1), i\geq 1  \text{ in } B_r\geq \lfloor \log r \rfloor-1,$$
which finishes the proof of the first case.

To prove the second case, we firstly denote by $r_0> 2$ the unique real number such that $\delta = C_0 r_0^\frac{\log r_0 -1}{2}.$ By assumption $r_0\leq r$ and we denote by $k\geq 1$ an integer such that $2^k \leq r_0 < 2^{k+1}.$ Now $\delta < C_0 2^{\frac{k(k+1)}{2}}$ and Lemma \ref{Lemma_Separated} implies that $\{ |F| \leq \delta \}$ does not intersect hyperplanes $\{Re(z)=2^i+2^{i-1} \}, i\geq k+1.$ As in the first case, zeros $(2^i,1),i\geq k+1$ belong to different connected components of $\{ |F|\leq \delta \}$ and thus
$$\zeta(F,r,\delta)\geq \text{the number of points }  (2^i,1), i\geq k+1  \text{ in } B_r\geq \lfloor \log r \rfloor -1-k.$$
Since $\log \delta =\log C_0 + \frac{1}{2}\log r_0 (\log r_0 -1) > \log C_0 + \frac{1}{2}(k-1)^2$, we have that $k<\sqrt{2\log \delta - 2\log C_0}+1$ and the claim follows. 
\end{proof}
\section{Counting islands}\label{sec: counting islands}
Let us start with the proof of Theorem \ref{prop: rouche} which, as was mentioned in the introduction, is a an easy corollary of Rouch\'e's theorem.
\begin{proof}[Proof of Theorem \ref{prop: rouche}]
 By Rouch\'e's theorem for analytic mappings from $\C^n$ to $\C^n$ (see e.g. \cite[Section 2.1.3]{DAngelo}) if $g:\C^n \to \C^n$ is a polynomial mapping of degree at most $k$ such that $d_{C^0}(f|_{B_r},g|_{B_r}) < \delta/2,$ then \[\tau(f,r,\delta) \leq \tau(g,r,\delta/2).\] By B\'ezout's theorem, however, \[\tau(g,r,\delta/2) \leq k^n.\] By Proposition \ref{lma: Cauchy harmonic}, it is enough to take $k$ such that $C_a a^{-k} \mu(f,ar) < \delta/2.$ It is easy to see that the optimal such $k$ satisfies \[k \leq C'_a \log(\mu(f,ar)/\delta).\] Combining the three displayed inequalities finishes the proof. 
\end{proof}
Before we give a formal proof of Theorem \ref{Proposition_Islands_Count}, let us briefly explain the geometric intuition as we did for the proof of Theorem \ref{Thm:CS_Coarse_Count}. As we already explained, for a fixed $\delta$, the sublevel set $\{ |F| \leq \delta \}$ stabilizes starting from $r=2^{k_0}$ into components which contain intervals $\{2^k \} \times [0,1]$ , $k\geq k_0$, but which are separated by hyperplanes $H_k=\{ Re(z)=2^k+2^{k-1} \}. $ However, we will show that these components in fact contain intervals $\{2^k \} \times [0,L(k)]$ where $L(k)$ can grow arbitrarily fast, the lower bound on growth depending on $\mathfrak{c}.$ This follows from Lemma \ref{Lemma_Peninsula}. In other words, components of $\{|F|\leq \delta \}$ get elongated in $w$-direction and thus they partly remain outside of $B_r$ for very large $r$, as shown on Figure \ref{Figure_Zeros_Islands}. Due to this elongation, most of the components of $\{|F|\leq \delta \}$ only contribute to $\zeta(F,r,\delta)$ and not to $\zeta^0(F,r,\delta)$, which leads to the upper bound given by Theorem \ref{Proposition_Islands_Count}.

\begin{figure}[ht]
	\begin{center}
		\includegraphics[scale=0.80]{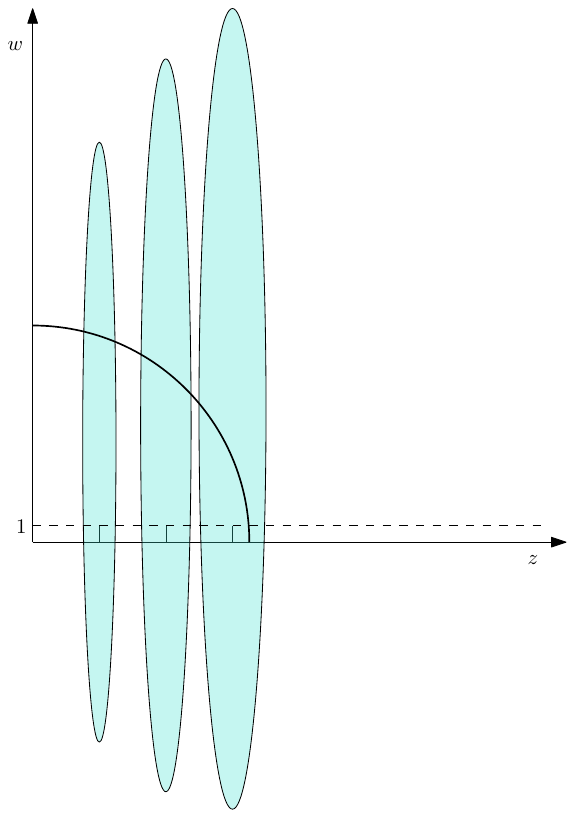}
		\caption{Components count for $\zeta(F,r,\delta)$, but not for $\zeta^0(F,r,\delta)$}
		\label{Figure_Zeros_Islands}
	\end{center}
\end{figure}

\begin{proof}[Proof of Theorem \ref{Proposition_Islands_Count}]
Since $\zeta^0(F,r,\delta)$ is decreasing in $\delta$, it is enough to prove the statement for $\delta \leq 1.$ Denote by $b_i= 2^{c_i-\frac{(i-1)i}{2c_i}}\delta^{\frac{1}{c_i}}.$ Let $i_0(l,\lambda,\delta)$ be the smallest index such that
\begin{itemize}
\item[1)] For all $i\geq i_0$, $\delta \geq \max (2^{-c_i},2^{\frac{(i-1)i}{2}-c_i^2})$;
\item[2)] $\log b_{i_0}\geq i_0$;
\item[3)] For all   $i\geq i_0$, $b_i$ is increasing.
\end{itemize}
Now,
\begin{equation}\label{Estimate_Islands}
\zeta^0(F,r,\delta) \leq \sum_{1\leq i \leq \log b_{i_0}} c_i + \textrm{I},
\end{equation}
where 
$$\textrm{I} = \text{ the number of islands with zeros from } \{ 2^i \} \times [0,1], i >\log b_{i_0}.$$
Since the first term of the right-hand side of this inequality depends only on $\delta$ and $\mathfrak{c}$, we wish to estimate $\textrm{I}$.  If $r\leq b_{i_0}$ then $2^i>r$ and $\textrm{I}=0$. Thus, we assume that $r>b_{i_0}.$ Firstly, from 2) it follows that
\begin{equation}\label{Estimate_Islands_2}
\textrm{I} \leq \text{ the number of islands with zeros from } \{ 2^i \} \times [0,1], i >i_0.
\end{equation}
Now, 3) implies that there exists a unique integer $k\geq i_0$ such that $b_k\leq r <b_{k+1}.$ Due to 1) we may apply Lemma \ref{Lemma_Peninsula} to conclude that $\{ |F|\leq \delta \}$ contains intervals $\{2^i \} \times [0,b_i]$ for all $i\geq i_0.$ This fact, combined with (\ref{Estimate_Islands_2}) implies
$$\textrm{I} \leq \text{ the number of intervals } \{ 2^i \} \times [0,b_i], i >i_0 \text{ contained in } B_r \leq k-i_0.$$
Going back to (\ref{Estimate_Islands}) we have that
\begin{equation}\label{Estimate_Islands_3}
\zeta^0(F,r,\delta) \leq \sum_{1\leq i \leq \log b_{i_0}} c_i -i_0 +k.
\end{equation}
Finally, 1) gives us that $\delta^{\frac{1}{c_k}}\geq \frac{1}{2}$ and hence $2^{c_k-\frac{(k-1)k}{2c_k}-1}\leq b_k\leq r.$ Taking logarithms we obtain that 
$$c_k\leq \log r+\frac{(k-1)k}{2c_k}+1.$$
Since $\frac{(k-1)k}{2c_k}+1$ has an upper bound which only depends on $\lambda$, taking logarithms $l$ times gives us
$$\underbrace{\log \ldots \log}_{l \text{ times}} c_k \leq \underbrace{\log \ldots \log}_{l+1 \text{ times}} r + a_{\lambda},$$
where $a_{\lambda}$ depends only on $\lambda.$ Substituting the desired value of $c_k$ in this inequality together with (\ref{Estimate_Islands_3})  finishes the proof.
\end{proof}



\section{Comparison with other results}\label{sec: CA example}

The goal of this section is to compare the results of this paper to the results of \cite{LiTaylor-TB}. More precisely, we will deduce Theorem 5.1 in \cite{LiTaylor-TB} from Theorem \ref{prop: rouche}, as well as show that Theorem \ref{thm: analytic bezout} does not follow from Theorem 5.1 in \cite{LiTaylor-TB}. We start by recalling this result. 

For an entire map $f:\C^n\to \C^n$ let $J_f$ denote the complex Jacobian matrix. Given a sequence of zeros $\xi=\{ \xi_i\}\subset f^{-1}(0)$ we define $\zeta_\xi(f,r)$ to be the number of elements of $\xi$ inside a ball $B_r.$

\begin{theorem}[Theorem 5.1 in \cite{LiTaylor-TB}]\label{thm: Li Taylor}
Let $f:\C^n\to \C^n$ be an entire map and $\xi=\{ \xi_i\}$ a sequence of zeros of $f.$ If there exist real numbers $c>0$ and $b$ such that
$$(\forall i) ~~~ |\det J_f(\xi_i)| \geq c (\mu(f,|\xi_i|))^{-b},$$
then for any $a>1$ it holds
$$\zeta_{\xi}(f,r)=O((\log \mu(f,ar))^n),$$
when $r\to \infty.$
\end{theorem}

Firstly, we give a proof of Theorem \ref{thm: Li Taylor} using Theorem \ref{prop: rouche}. The strategy of the proof follows \cite{Ji}. Namely,  the main results of \cite{Ji}, Theorems 1.1 and 1.2, are proven using a lemma which was referred to as ``Weak B\'{e}zout estimate", see \cite[Lemma 3.1]{Ji}. This lemma establishes an inequality
\begin{equation}\label{Ji_Coarse}
\tau(f,r,\delta) \leq C_n \left( (r+1)\frac{\mu(f,r+1)}{\delta} \right)^{2n},
\end{equation}
with $C_n$ which depends only on $n.$ The proof of (\ref{Ji_Coarse}) relies on a global version of the Chern-Levine-Nirenberg inequality, see \cite[Theorem 2.1]{Ji} and references therein.  Substituting (\ref{Ji_Coarse}) with Theorem \ref{prop: rouche} and using the same general arguments as in \cite{Ji} proves Theorem \ref{thm: Li Taylor}.  To implement this strategy, we will need the following lemma.

\begin{lemma}\label{Lemma_Ji} Let $f:\C^n\to \C^n$ be an entire map and $\xi$ a zero of $f$ such that $J_f(\xi)\neq 0.$ Then, for all $z\in \C^n$ such that
$$|z| \leq \frac{1}{2\left( n! \frac{(\mu(f,|\xi|+1))^n}{|\det J_f(\xi)|} +1 \right)}$$
it holds
$$|f(\xi+z)|\geq \frac{ |\det J_f(\xi)| \cdot |z|}{2n! (\mu(f,|\xi|+1))^{n-1}}.$$
\end{lemma}

The proof of Lemma \ref{Lemma_Ji} can be extracted from the proof of Theorem 1.1 in \cite{Ji}. We present it here for the sake of completeness. We will use the following auxiliary statement, which is a direct consequence of Schwarz lemma, see \cite[Lemma 3.4]{Ji} for details.

\begin{lemma}\label{Schwarz}
Let $f:\C^n\to \C^n$ be an entire map such that $f(0)=0, J_f(0)=0.$ Then for all $z$ such that $|z|\leq \frac{1}{2\mu(f,1)}$ it holds $|f(z)|\leq \frac{1}{2}|z|.$
\end{lemma}
\begin{proof}[Proof of Lemma \ref{Lemma_Ji}]
We start by proving the following auxiliary inequality
\begin{equation}\label{Inequality_Inverse}
\| J_f(\xi)^{-1} \|_{op} \leq n! \frac{(\mu(f,|\xi|+1))^{n-1}}{|\det J_f(\xi)|},
\end{equation}
where $\| \cdot \|_{op}$ denotes the operator norm. Recall that
$$\| J_f(\xi)^{-1}\|_{op}=\frac{1}{|\det J_f(\xi) |} \| \adj(J_f(\xi))\|_{op},$$
and thus we need to prove that 
$$\| \adj(J_f(\xi))\|_{op} \leq n! (\mu(f,|\xi|+1))^{n-1}.$$
By Cauchy-Schwarz inequality
$$\| \adj(J_f(\xi))\|_{op} \leq n \cdot \max_{i,j} |\adj(J_f(\xi))_{i,j}|,$$
and we are left to prove that 
$$\max_{i,j} |\adj(J_f(\xi))_{i,j}|\leq (n-1)! (\mu(f,|\xi|+1))^{n-1}.$$
For each $i$ and $j$, $\adj(J_f(\xi))_{i,j}$ is a sum of $(n-1)!$ terms, each of which is a product of $n-1$ partial derivatives of $f.$ Thus
$$\max_{i,j} |\adj(J_f(\xi))_{i,j}| \leq (n-1)! \cdot (\max_i |\partial_if(\xi)|)^{n-1}.$$
Finally, Cauchy's inequality yields that $\max_i |\partial_if(\xi)|\leq \mu (f,|\xi|+1)$ which completes the proof of (\ref{Inequality_Inverse}).

Now, let $g:\C^n \to \C^n$ be an entire map given by $g(z)=(J_f(\xi))^{-1}f(\xi+z).$ Since $g(0)=0$ and $J_g(0)=\id_{\C^n}$ we may apply Lemma \ref{Schwarz} to the map $g(z)-z$, which gives us
\begin{equation}\label{Schwarz_Triangle_1}
|g(z)-z|\leq \frac{1}{2}|z|,
\end{equation}
for all $z$, such that $|z|\leq \frac{1}{2\mu(g(z)-z, 1)}$. Moreover, $\mu(g(z)-z,1)\leq \mu(g,1)+1$ implies that (\ref{Schwarz_Triangle_1}) holds for all $z$ with $|z| \leq \frac{1}{2(\mu(g,1)+1)}$ and triangle inequality further implies that
\begin{equation}\label{Schwarz_Triangle_2}
|g(z)| \geq \frac{1}{2}|z|
\end{equation}
as long as $|z|\leq \frac{1}{2(\mu(g,1)+1)}.$ From the definition of $g$ and (\ref{Schwarz_Triangle_2}) it follows that
\begin{equation}\label{Schwarz_Triangle_3}
|f(\xi+z)| \cdot \| (J_f(\xi))^{-1} \|_{op} \geq \frac{1}{2}|z|,
\end{equation}
for all $z$ such that $|z|\leq \frac{1}{2(\mu(g,1)+1)}.$ Applying (\ref{Inequality_Inverse}) to (\ref{Schwarz_Triangle_3}) yields
\begin{equation}\label{Schwarz_Triangle_4}
 |f(\xi+z)| \cdot n! \frac{(\mu(f,|\xi|+1))^{n-1}}{|\det J_f(\xi)|} \geq \frac{1}{2}|z|,
\end{equation}
for all $z$ such that $|z|\leq \frac{1}{2(\mu(g,1)+1)}.$
Using (\ref{Inequality_Inverse}) again gives us
$$|g(z)|\leq |f(\xi+z)| \cdot \| (J_f(\xi))^{-1} \|_{op} \leq  |f(\xi+z)| \cdot n! \frac{(\mu(f,|\xi|+1))^{n-1}}{|\det J_f(\xi)|}, $$
and thus 
$$\mu(g,1)\leq n! \frac{(\mu(f,|\xi|+1))^n}{|\det J_f(\xi)|}.$$
Combining the last inequality and (\ref{Schwarz_Triangle_4}) finishes the proof.
\end{proof}

\begin{proof}[Proof of Theorem \ref{thm: Li Taylor}]
For simplicity we assume that $b\geq 0$ and $f(0)\neq 0.$ The general case readily reduces to this one. 

Let $r\geq 0$ and $\xi_i \in B_r.$ By the assumption
$$|\det J_f(\xi_i)|\geq c(\mu(f,|\xi_i|))^{-b} \geq c (\mu(f,r))^{-b} \geq c (\mu (f,r+1))^{-b},$$
and since $\mu(f,|\xi_i|+1)\leq \mu(f,r+1)$, Lemma \ref{Lemma_Ji} gives us that
\begin{equation}\label{Expression_For_Delta}
|f(\xi_i + z)|\geq \delta := \frac{c|z|}{2 n! (\mu(f,r+1))^{n-1+b}},
\end{equation}
for all $z$ such that
$$|z|\leq \varepsilon := \frac{1}{2( c^{-1} \cdot n! (\mu(f,r+1))^{n+b} +1)}.$$
Since $\varepsilon<\frac{1}{2}<1$ a connected component of $\{ |f|\leq \delta \}$ which contains $\xi_i$ is itself fully contained in $B_{r+1}.$ The same holds for each $\xi_j\in B_r$ and thus
$$\zeta_{\xi}(f,r)\leq \tau(f,r+1,\delta).$$
Finally,  by (\ref{Expression_For_Delta}), $\delta/2 \leq \mu(f,a(r+1))/2$ for all $a>1$ and we may apply Theorem \ref{prop: rouche} which yields
$$\zeta_{\xi}(f,r)\leq \tau(f,r+1,\delta) \leq C_{n,a} \left( \log \frac{2\cdot \mu(f,a(r+1))}{\delta} \right)^n.$$
Substituting the expression for $\delta$ into this inequality and using $\mu(f,r+1)\leq \mu (f,a(r+1))$ gives us
$$\zeta_{\xi}(f,r)\leq C_{n,a} \left( \log \frac{4n!\cdot (\mu(f,a(r+1)))^{n+b}}{c|z|} \right)^n$$
for all $z$ such that $|z|\leq \varepsilon.$ Taking $|z|=\varepsilon$ yields
$$\zeta_{\xi}(f,r) \leq C_{n,a} \left( A_{n,b} \log \mu(f,  a(r+1)) + B_{n,c} \right)^n,$$
where $A,B,C$ depend on $n,a,b,c$ as indicated. This completes the proof.
\end{proof}

So far we proved that Theorem \ref{thm: Li Taylor} follows from Theorem \ref{prop: rouche}.  Now, we wish to show that Theorem \ref{thm: analytic bezout} does not follow from Theorem \ref{thm: Li Taylor}. Namely, one may imagine the following scenario. Let $f:\C^n\to \C^n$ be an entire map. While the classical count of zeros of $f$ might not satisfy the transcendental B\'ezout bound, it may still happen that for a fixed $\delta>0$, we can choose a sequence of zeros $\{\xi_i\}$ from $f^{-1}(B_\delta)$ such that Theorem \ref{thm: Li Taylor} applies to this sequence and $\zeta(f,r,\delta)=O(\zeta_\xi (f,r)).$ In this case, Theorem \ref{thm: Li Taylor} would imply Theorem \ref{thm: analytic bezout} (at least up to a constant which depends on $\delta$). However,  our next result rules out this possibility.

\begin{prop}\label{prop: CS-Jacobian}
Let  $\mathfrak{c}=\{c_i \}$ be an increasing sequence of positive integers such that $\lim_{i\to +\infty} \frac{i}{c_i}=0$ and $F$ the corresponding Cornalba-Shiffman map. For all real numbers $c>0$ and $b$ the inequality
$$|\det J_f(\xi)| \geq c (\mu(f,|\xi|))^{-b},$$
holds for at most finitely many zeros $\xi\in F^{-1}(0).$
\end{prop} 

\begin{proof}
Since the zeros of Cornalba-Shiffman maps are isolated and $\mu(F,r)\to + \infty$ as $r\to +\infty$, it is enough to prove the proposition for $b>0.$ Since $F(z,w)=(g(z),f(z,w))$, we have that
$$\det J_F=\partial_z g\partial_w f-\partial_w g\partial_z f=\partial_z g\partial_w f.$$
All relevant infinite sums and products converge uniformly on compact sets and thus we may compute
$$\partial_z g = \sum_{l=1}^\infty -2^{-l}g_l,$$
as well as
$$\partial_w f = \sum_{l=1}^\infty 2^{-c_l^2}g_l\partial_w P_{c_l}.$$
We wish to evaluate $J_F$ at zeros of $F.$ To this end, let us denote the zeros of $F$ by $\xi_{i,j}=(2^i,1/j)$ for $i\geq 1$ and $1\leq j \leq c_i.$ We calculate
$$\partial_z g (2^i) = -2^{-i}g_i(2^i),$$
as well as
$$\partial_w f (\xi_{i,j}) = 2^{-c_i^2}g_i(2^i)\partial_w P_{c_i}(1/j)=2^{-c_i^2}g_i(2^i) \prod_{1\leq l\leq c_i, l\neq j} \left( \frac{1}{j}- \frac{1}{l} \right).$$
Since $\Big\vert \prod_{1\leq l\leq c_i, l\neq j} \left( \frac{1}{j}- \frac{1}{l} \right)\Big\vert \leq 1,$ we get that for all $i$ and $j$
\begin{equation}\label{Jacobian_Estimate}
|\det J_F(\xi_{i,j})| \leq 2^{-c_i^2-i}(g_i(2^i))^2.
\end{equation}
Moreover,
$$g(2^i)= \prod_{l=1}^{i-1} (1-2^{i-l}) \cdot  \prod_{l=i+1}^{\infty} (1-2^{i-l})= \prod_{l=1}^{i-1} (1-2^l) \cdot  \prod_{l=1}^{\infty} (1-2^{-l}), $$
and thus we have that
$$|g_i(2^i)|< \prod_{l=1}^{i-1} (2^l-1)<\prod_{l=1}^{i-1} 2^l=2^\frac{i(i-1)}{2}.$$
Combining this inequality with (\ref{Jacobian_Estimate}) yields
\begin{equation}\label{c_i_Jacobian}
|\det J_F(\xi_{i,j})|  < 2^{-c_i^2+i^2-2i},
\end{equation}
for all $i$ and $j.$ By the assumption on $\mathfrak{c}$, we have that for any $b>0$, $c_i^2\geq 5b(i+1)^2+i^2$ for all but finitely many $i.$ Thus, for all but finitely many $i$, it holds
\begin{equation}\label{Jacobian_Auxiliary_1}
|\det J_F(\xi_{i,j})|  < 2^{-5b(i+1)^2-2i}
\end{equation}
for all $j.$ Since $|\xi_{i,j}|=|(2^i,1/j)|<2^{i+1}$, we have that
\begin{equation}\label{Jacobian_Auxiliary_2}2^{-5b(i+1)^2}<2^{-5b(\log |\xi_{i,j}|)^2}<\left( 2^{\frac{3}{2}(\log |\xi_{i,j}|)^2+\frac{7}{2} \log |\xi_{i,j}|} \right)^{-b}. 
\end{equation}
Now, by Proposition \ref{Estimate_of_Max}
\begin{equation}\label{Jacobian_Auxiliary_3}
\left( 2^{\frac{3}{2}(\log |\xi_{i,j}|)^2+\frac{7}{2} \log |\xi_{i,j}|} \right)^{-b} \leq 2^{bC} (\mu(F,|\xi_{i,j}|))^{-b},
\end{equation}
where $C>0$ is an absolute constant. Putting (\ref{Jacobian_Auxiliary_1}), (\ref{Jacobian_Auxiliary_2}) and (\ref{Jacobian_Auxiliary_3}) together gives us
$$|\det J_F(\xi_{i,j})|  <2^{-2i+bC} (\mu(F,|\xi_{i,j}|))^{-b}$$
for all but finitely many indices $i,j.$ Since for every $a>0,2^{-2i+bC}<a$ for all but finitely many $i$, the proof follows.
\end{proof}

\begin{rmk}
Different results in the spirit of Theorem \ref{thm: Li Taylor} were obtained in \cite{Ji} and \cite{Li-TB}. In one way or another, these results rely on lower bounds on $\det J_f$ at zeros of $f.$ Namely, in \cite{Ji}, $|\det J_f(\xi_i)|$ is assumed to be bounded from below by a constant, while in \cite{Li-TB} the upper bound for $\zeta_\xi(f,r)$ involves terms of the form $\log \frac{1}{\det J_f(\xi_i)}.$ From (\ref{c_i_Jacobian}), it follows that by taking $\{c_i\}$ which increases sufficiently fast, we can make $|\det J_F(\xi_{i,j})|$ of Cornalba-Shiffman maps decrease arbitrarily fast. Since, by Theorem \ref{Thm:CS_Coarse_Count}, $\zeta(F,r,\delta)$ increases as $\log r$ independently of $\{c_i\}$, we conclude that Theorem \ref{thm: analytic bezout} can not be deduced from the results of \cite{Ji} and \cite{Li-TB} using the above-described strategy. 
\end{rmk}

\begin{rmk}
It is interesting to notice that Proposition \ref{prop: CS-Jacobian} does not rule out a possibility that Theorem \ref{prop: rouche}, or at least the same bound for $\zeta^0$, can be deduced from Theorem \ref{thm: Li Taylor}. Indeed, we have not proven any lower bound on the count of islands of Cornalba-Shiffman maps. As a matter of fact, it is not even clear if for each $\delta>0$ and each sequence $\{c_i\}$, $\zeta^0(F,r,\delta)\to +\infty$ as $r\to +\infty.$ Namely, it may happen that starting from certain finite $r_0$, all connected components of $\{ |F|\leq \delta \}$, which contain zeros of $F$,  elongate all the way to infinity in the $w$-direction and thus never become islands, but rather remain peninsulas for all $r>r_0.$
\begin{question}
Is it true that for each $\delta>0$ and all sequences $\mathfrak{c}$, $\zeta^0(F,r,\delta)\to +\infty$ as $r \to +\infty$? If so, what is the possible growth rate of $\zeta^0(F,r,\delta)$ depending on parameters $\delta$ and $\mathfrak{c}$?
\end{question}
\end{rmk}

\begin{rmk}
The main technical ingredient in \cite{LiTaylor-TB} and \cite{Li-TB} is Theorem 3.6 from \cite{LiTaylor-TB} (slightly modified in \cite{Li-TB}). While the proof of this result has certain similarities with the proof of Theorem \ref{prop: rouche}, it seems that the two approaches are fundamentally different. Namely, approximation of a holomorphic map by a polynomial is the key idea in the proof of Theorem \ref{prop: rouche}, while it is not directly used in the proof of Theorem 3.6 in \cite{LiTaylor-TB}. It would be interesting to explore how the methods of \cite{LiTaylor-TB} and \cite{Li-TB} relate to the coarse counts of zeros. In the opposite direction, it would be interesting to deduce the results of \cite{Li-TB} using the same strategy as above, by proving a suitable analogue of Theorem \ref{prop: rouche}.
\end{rmk}


\section{Coarse counts and persistent homology}\label{sec: persistence}

In this section we discuss results generalizing Theorems \ref{thm: analytic bezout} and \ref{thm: analytic bezout higher homol}, formulated in terms of persistent homology and barcodes.

\subsection{Persistence modules and barcodes}\label{subsec:intro bars}

Recall that for a Morse function $f:M \to \R$ on a compact manifold and a coefficient field $\bb  K$, its {\em barcode} in degree $\m \in \Z$ is a finite multi-set $\cl B_\m(f;\bb K)$ of intervals with multiplicities $(I_j,m_j),$ where $m_j \in \N$ and $I_j $ is finite, that is of the form $[a_j,b_j)$ or infinite, that is of the form $[c_j,\infty).$ The number of infinite bars is equal to the Betti number $b_\m(M;\bb K) = \dim H_\m(M;\bb K).$

This barcode is obtained algebraically from the {\em persistence module} $V_\m(f;\bb K)$ consisting of vector spaces $V_\m(f;\bb K)^t = H_\m(\{f \leq t\}; \bb K)$ parametrized by $t\in \R$ and connecting maps $\pi^{s,t}: V_\m(f;\bb K)^s \to V_\m(f;\bb K)^t$ induced by the inclusions $\{f \leq s \} \hookrightarrow \{f \leq t \}$ for $s\leq t.$ These maps satisfy the structure relations of a persistence module: $\pi^{s,s} =\id_{V_\m(f;\bb K)^s}$ for all $s$ and $\pi^{s_2,s_3} \circ \pi^{s_1,s_2} = \pi^{s_1,s_3}$ for all $s_1 \leq s_2 \leq s_3.$ The total barcode of $f$ is set to be \[\cl B(f;\bb K) = \sqcup_{\m \in \Z} \cl B_\m(f;\bb K)\] where $\sqcup$ stands for the sum operation on multisets. This is the barcode of the persistence module \[V(f;\bb K) = \oplus_{\m \in \Z} V_\m(f;\bb K).\] 

On a compact manifold $M$ with boundary $\partial M$ and a Morse function $f: M \to \R$ in the sense of manifolds with boundary, we may define the persistence module and barcode of $f$ as above. 


One simple property of the barcode of this persistence module is that 
the number of bars in the barcode coincides with the number of their starting points. Another property is that the number of bars containing a given interval $[a,b]$ is $\dim \mrm{Im}(\pi^{a,b}).$


Recall that the length of a finite bar $[a,b)$ is $b-a$ and the length of an infinite bar $[c,\infty)$ is $+\infty.$ We require the following notion: for $\delta \geq 0,$ let $\cl N_{\delta}(f;\bb K)$ denote the number of bars of length $>\delta$ in the barcode $\cl B(f;\bb K).$ Similarly, $\cl N_{\m,\delta}(f;\bb K)$ is the number of bars of length $>\delta$ in the barcode $\cl B_\m(f;\bb K)$ and $\cl N_{\delta}(f;\bb K) = \sum_\m \cl N_{\m,\delta}(f;\bb K).$ The definitions of $\cl N_{\m,\delta}(f;\bb K)$ and  $\cl N_{\delta}(f;\bb K)$ extend to any continuous function $f$ on a compact manifold with boundary, as explained in detail in \cite[Section 2.2]{BPPPSS}.


We refer to \cite{PRSZ} for a systematic introduction to persistence modules with a view towards applications in topology and analysis. The only result which we require here is the following direct consequence of the algebraic isometry theorem \cite{Bauer-Lesnick} (see \cite[Theorem 2.2.8, Equation (6.4)]{PRSZ}).

\begin{theorem}\label{thm: stab}
Let $f,g: M \to \R$ be two functions on a compact manifold $M$ with boundary such that $d_{C^0}(f,g) \leq c-\eps/2$ for $c > \eps/2 > 0.$ Then for all $\m \in \Z,$ $\cl N_{\m,2c}(f) \leq \cl N_{\m,\eps}(g).$
\end{theorem}

\subsection{Bounds on barcodes of analytic functions}

For a continuous map $f: \C^n \to \C^m,$ and $r>0,$ set \[ \cl N_{\delta}(f,r) = \cl N_{\delta}( |f | \big |_{B_r}),\;\;\;  \cl N_{q,\delta}(f,r) = \cl N_{q, \delta}(|f| \big |_{B_r}).\] 
Consider the invariants $\zeta_d(f,r,\delta)$ from Section \ref{sec: discussion zeta m}. Note that \[ \zeta_d(f,r, \delta) \leq \cl N_{\delta}(f,r),\;\;   \zeta(f,r,\delta) \leq \cl N_{0,\delta}(f,r),\] 
as shown in Lemma 6.4 and Remark 6.5 in \cite{BPPPSS}. Using the terminology and methods of persistence we can prove the following result, which therefore generalizes Theorems \ref{thm: analytic bezout} and \ref{thm: analytic bezout higher homol}.

\begin{theorem}\label{thm: analytic bezout persistence}
For any analytic map $f: \C^n \to \C^m$, $m\leq n$ and any  $a>1$, $r>0$, and $\delta \in (0, \frac{\mu(f, ar)}{2})$,  we have

\begin{equation}
\label{eq: main bezout pers}
\cl N_{\delta}(f) \leq C \left(\log\left(\frac{\mu(f, ar)}{\delta}\right)\right)^{2n},
\end{equation}

\begin{equation}
\label{eq: main bezout pers 0}
\cl N_{0, \delta}(f) \leq C \left(\log\left(\frac{\mu(f, ar)}{\delta}\right)\right)^{2n-1},
\end{equation}
where the constant $C$ depends only on $a$ and $n$. 	

\end{theorem}

Since, given Theorem \ref{thm: stab}, this result reduces essentially to our proofs above and does not
influence the main results of the paper, we only briefly sketch its proof. 

\begin{proof}[\bf Sketch of the proof.]
Firstly, we reduce as above to the case $m=n,$ see Remark \ref{rmk: Dimension_Reduction}. By Proposition \ref{lma: Cauchy harmonic}, we can approximate $ f $ by a complex Taylor polynomial mapping $ p $ at $ 0 $ of degree $ <k $ such that $ |f-p| \leq C_{a} a^{-k} \mu(f,ar) $ on $ B_r $. Here we choose $ k $ to be the minimal positive integer such that $ C_a a^{-k} \mu(f,ar) < \delta/2 $. Furthermore, by a classical but lengthy transversality argument (which we omit), we may assume that $h=|p|^2$ has no critical points on $\del B_r$, is Morse on $B_r$  and $h|_{\del B_r}$ is Morse. Now, by Proposition \ref{prop: complex poly} the number of critical points of $h$ on $B_r$ is at most $C k^{2n},$ the number of critical points of $h$ of index $0$ is at most $C k^n,$ while the number of critical points of $h|_{\del B_r}$ is at most $C k^{2n-1}.$ The starting points of bars of degree $0$ correspond to critical points of $h$ of index $0$ and to critical points of index $0$ of $h|_{\del B_r}$ with gradient pointing inwards, while the endpoints of bars in general correspond to a subset of the critical points of $h$ and of $h|_{\del B_r}.$ This follows from Morse theory for manifolds with boundary and was discussed in a more general framework in \cite[Proposition 4.12]{BPPPSS}. Therefore $\cl N_{\eps}(|p|) \leq C k^{2n}$ and $\cl N_{0,\eps}(|p|) \leq C k^{2n-1}$ for any $\eps>0.$ Taking $ ||f|-|p|| \leq |f-p| < c< \delta/2$ and $0<\eps<c - \delta$,  Theorem \ref{thm: stab} implies $\cl N_{\delta}(|f|) \leq C k^{2n}$ and $\cl N_{0,\delta}(|f|) \leq C k^{2n-1},$ which translates to the required bound by our choice of $k.$
\end{proof}

\begin{rmk}
An entire map $f:\C^n\to \C^n$ naturally gives rise to a persistence module $H_*(\{ |f|\leq t \} \cap B_r )$ in two parameters $r$ and $t.$ In this paper we considered it as an $r$-parametrized family of persistence modules with one parameter $t$. It would be interesting to study this persistence module from the viewpoint of multiparameter persistence, see \cite{BotnanLesnick} and references therein, for example by using the recently introduced language of signed barcodes, see \cite{BOO1,BOO2,OudotScoccola}.
\end{rmk}

\bibliographystyle{abbrv}
\bibliography{bibliographySobolev0423.bib}

\begin{thebibliography}{10}

\bibitem{Ahlfors}
L.~V. Ahlfors.
\newblock {\em Complex analysis}.
\newblock McGraw-Hill Book Co., New York, third edition, 1978.

\bibitem{Arosio-etal}
L.~Arosio, A.~M. Benini, J.~E. Forn{\ae}ss, and H.~Peters.
\newblock Dynamics of transcendental {H}\'{e}non maps {III}: Infinite entropy.
\newblock {\em J. Mod. Dyn.}, 17:465--479, 2021.

\bibitem{Artin-Mazur}
M.~Artin and B.~Mazur.
\newblock On periodic points.
\newblock {\em Ann. of Math. (2)}, 81:82--99, 1965.

\bibitem{HFT}
S.~Axler, P.~Bourdon, and W.~Ramey.
\newblock {\em Harmonic function theory}, volume 137 of {\em Graduate Texts in
  Mathematics}.
\newblock Springer-Verlag, New York, second edition, 2001.

\bibitem{Bauer-Lesnick}
U.~Bauer and M.~Lesnick.
\newblock Induced matchings and the algebraic stability of persistence
  barcodes.
\newblock {\em J. Comput. Geom.}, 6(2):162--191, 2015.

\bibitem{BotnanLesnick}
M.~B. Botnan and M.~Lesnick.
\newblock An introduction to multiparameter persistence.
\newblock Preprint, arXiv:2203.14289, 2023.

\bibitem{BOO1}
M.~B. Botnan, S.~Oppermann, and S.~Oudot.
\newblock Signed barcodes for multi-parameter persistence via rank
  decompositions.
\newblock In {\em 38th {I}nternational {S}ymposium on {C}omputational
  {G}eometry}, volume 224 of {\em LIPIcs. Leibniz Int. Proc. Inform.}, pages
  Art. No. 19, 18. Schloss Dagstuhl. Leibniz-Zent. Inform., Wadern, 2022.

\bibitem{BOO2}
M.~B. Botnan, S.~Oppermann, S.~Oudot, and L.~Scoccola.
\newblock On the bottleneck stability of rank decompositions of multi-parameter
  persistence modules.
\newblock Preprint, arXiv:2208.00300, 2022.

\bibitem{BPPPSS}
L.~Buhovsky, J.~Payette, I.~Polterovich, L.~Polterovich, E.~Shelukhin, and
  V.~Stojisavljevi\'{c}.
\newblock Coarse nodal count and topological persistence.
\newblock Preprint, arXiv:2206.06347, 2022.

\bibitem{Carlson-TB}
J.~A. Carlson.
\newblock A moving lemma for the transcendental {B}ezout problem.
\newblock {\em Ann. of Math. (2)}, 103(2):305--330, 1976.

\bibitem{Chow49}
W.-L. Chow.
\newblock On compact complex analytic varieties.
\newblock {\em American Journal of Mathematics}, 71(4):893--914, 1949.

\bibitem{Cornalba-Griffiths}
M.~Cornalba and P.~Griffiths.
\newblock Analytic cycles and vector bundles on non-compact algebraic
  varieties.
\newblock {\em Invent. Math.}, 28:1--106, 1975.

\bibitem{CornalbaShiffman}
M.~Cornalba and B.~Shiffman.
\newblock A counterexample to the ``transcendental {B}\'{e}zout problem''.
\newblock {\em Ann. of Math. (2)}, 96:402--406, 1972.

\bibitem{DAngelo}
J.~P. D'Angelo.
\newblock {\em Several complex variables and the geometry of real
  hypersurfaces}.
\newblock Studies in Advanced Mathematics. CRC Press, Boca Raton, FL, 1993.

\bibitem{Engelking89}
R.~Engelking.
\newblock {\em General topology}, volume~6 of {\em Sigma Series in Pure
  Mathematics}.
\newblock Heldermann Verlag, Berlin, second edition, 1989.
\newblock Translated from the Polish by the author.

\bibitem{Evans}
L.~C. Evans.
\newblock {\em Partial differential equations}, volume~19 of {\em Graduate
  Studies in Mathematics}.
\newblock American Mathematical Society, Providence, RI, second edition, 2010.

\bibitem{Fulton-intersection}
W.~Fulton.
\newblock {\em Intersection theory}, volume~2 of {\em Ergebnisse der Mathematik
  und ihrer Grenzgebiete. 3. Folge. A Series of Modern Surveys in Mathematics}.
\newblock Springer-Verlag, Berlin, second edition, 1998.

\bibitem{GLO}
A.~A. Gol’dberg, B.~Y. Levin, and I.~Ostrovskii.
\newblock {\em Entire and meromorphic functions}.
\newblock Springer, 1997.

\bibitem{Griffiths-TB2}
P.~A. Griffiths.
\newblock Function theory of finite order on algebraic varieties. {I}({B}).
\newblock {\em J. Differential Geometry}, 7:45--66, 1972.

\bibitem{Gutman-etal}
Y.~Gutman, Y.~Qiao, and M.~Tsukamoto.
\newblock Application of signal analysis to the embedding problem of {$\Bbb
  Z^k$}-actions.
\newblock {\em Geom. Funct. Anal.}, 29(5):1440--1502, 2019.

\bibitem{HeleinWood}
F.~H\'{e}lein and J.~C. Wood.
\newblock Harmonic maps.
\newblock In {\em Handbook of global analysis}, pages 417--491, 1213. Elsevier
  Sci. B. V., Amsterdam, 2008.

\bibitem{Ji}
S.~Ji.
\newblock B\'{e}zout estimate for entire holomorphic maps and their
  {J}acobians.
\newblock {\em Amer. J. Math.}, 117(2):395--403, 1995.

\bibitem{Li-TB}
B.~Q. Li.
\newblock On the {B}\'{e}zout problem and area of interpolating varieties in
  {$\bold C^n$}. {II}.
\newblock {\em Amer. J. Math.}, 120(6):1191--1198, 1998.

\bibitem{LiTaylor-TB}
B.~Q. Li and B.~A. Taylor.
\newblock On the {B}\'{e}zout problem and area of interpolating varieties in
  {$\bold C^n$}.
\newblock {\em Amer. J. Math.}, 118(5):989--1010, 1996.

\bibitem{Logunov-YauLower}
A.~Logunov.
\newblock Nodal sets of {L}aplace eigenfunctions: proof of {N}adirashvili's
  conjecture and of the lower bound in {Y}au's conjecture.
\newblock {\em Ann. of Math. (2)}, 187(1):241--262, 2018.

\bibitem{Milnor-book}
J.~Milnor.
\newblock {\em Morse theory}.
\newblock Annals of Mathematics Studies, No. 51.

\bibitem{Milnor}
J.~Milnor.
\newblock On the {B}etti numbers of real varieties.
\newblock {\em Proc. Amer. Math. Soc.}, 15:275--280, 1964.

\bibitem{Non19}
F.~Nonez.
\newblock Bornes sur les nombres de {B}etti pour les fonctions propres du
  {L}aplacien.
\newblock {\em M.Sc. thesis, Universit\'e de Montr\'eal}, 2020.

\bibitem{OudotScoccola}
S.~Oudot and L.~Scoccola.
\newblock On the stability of multigraded {B}etti numbers and {H}ilbert
  functions.
\newblock Preprint, arXiv:2112.11901, 2023.

\bibitem{PRSZ}
L.~Polterovich, D.~Rosen, K.~Samvelyan, and J.~Zhang.
\newblock {\em Topological Persistence in Geometry and Analysis}, volume~74 of
  {\em University Lecture Series}.
\newblock American Mathematical Society, Providence, 2020.

\bibitem{Rudin80}
W.~Rudin.
\newblock {\em Function theory in the unit ball of {$\Bbb C^n$}}.
\newblock Classics in Mathematics. Springer-Verlag, Berlin, 2008.
\newblock Reprint of the 1980 edition.

\bibitem{Serre}
J.-P. Serre.
\newblock G{\'e}om{\'e}trie analytique et g{\'e}om{\'e}trie alg{\'e}brique.
\newblock {\em Ann. Inst. Fourier, VI (1955--56)}, pages 1--42, 1955.

\end{thebibliography}

\end{document}